\documentclass[12pt, reqno]{amsart}
\usepackage{amsfonts,mathrsfs,amsmath,amssymb}

\usepackage[dvips]{color}
\usepackage{multicol}
\usepackage[all]{xy}           %xypic macro for latex2.09
\usepackage{graphicx}
\usepackage{color}
\usepackage{colordvi}

\theoremstyle{plain}

\numberwithin{equation}{section}

\newtheorem{theo}[equation]{Theorem}
\newtheorem{pro}[equation]{Proposition}

\newtheorem{lem}[equation]{Lemma}
\newtheorem{defi}[equation]{Definition}
\newtheorem{rem}[equation]{Remark}
\newtheorem{exa}[equation]{Example}

\newtheorem*{lemma}{Lemma}

\def\ot{\otimes}

\def\lan{\langle}
\def\ran{\rangle}

\def\Ga{\Gamma}

\def\ep{\epsilon}
\def\de{\delta}
\def\pa{\partial}

\def\c{{\bf c}}

\newcommand{\E}{{\mathcal{E}}}

\newcommand{\D}{{\mathcal D}}

\newcommand{\K}{{\mathcal K}}

\newcommand{\A}{{\mathcal A}}

\def\bN{{\mathbb N}}
\def\bZ{{\mathbb Z}}

\def\bC{{\mathbb C}}

\def\mg{\mathfrak{g}}

\def\Res{\mbox{\rm Res}}

\def\Hom{\mbox{\rm Hom}}

\def\Ker{\mbox{\rm Ker}\,}

\def\End{\text{\rm End}}

\begin{document}
\title[]
{$\phi_\ep$-coordinated modules for vertex algebras}

\author{Chengming Bai}
\address{Chern Institute of Mathematics \& LPMC, Nankai University, Tianjin 300071, P.R. China}
         \email{baicm@nankai.edu.cn}

\author{Haisheng Li}
\address{Department of Mathematical Sciences, Rutgers University, Camden, NJ 08102, United States}
\email{hli@camden.rutgers.edu}

\author{Yufeng Pei$^*$}
\address{Department of Mathematics, Shanghai Normal University, Guilin Road 100,
Shanghai 200234, P.R. China}\email{pei@shnu.edu.cn}

\thanks{$\mbox{}^*$ Corresponding author}

%\date{\today}

\begin{abstract}
We study $\phi_{\ep}$-coordinated modules for vertex algebras, where $\phi_{\ep}$ with $\ep$ an integer parameter is a family of associates of the one-dimensional additive formal group.
As the main results, we obtain a Jacobi type identity and a commutator formula for $\phi_{\ep}$-coordinated modules.
We then use these results to study $\phi_{\ep}$-coordinated modules for vertex algebras associated to
Novikov algebras by Primc.
\end{abstract}
\maketitle

%\tableofcontents

%\setcounter{section}{0}

%%%%%%%%%%%%%%%%%%%%%%%%%%
\section{Introduction}
%%%%%%%%%%%%%%%%%%%%%%%%%%

In a program to associate quantum vertex algebras to quantum affine algebras, a theory of $\phi$-coordinated (quasi) modules for quantum vertex algebras was developed in \cite{L6}, where $\phi$ is what was called an associate of the one-dimensional additive formal group
(law) $F_{\rm a}(x,y)=x+y$.

Since the very beginning, it had been recognized that the theory of vertex algebras and their modules
was governed by the formal group $F_{\rm a}$. This can be seen from the definition of a vertex algebra $V$ and more generally the definition of a module $(W,Y_{W})$ for a given vertex algebra $V$,
where the weak associativity axiom for a $V$-module $(W,Y_{W})$ states that
for any $u,v\in V,\ w\in W$, there exists a nonnegative integer $l$ such that
\begin{eqnarray}
(x+y)^{l}Y_{W}(u,x+y)Y_{W}(v,y)w=(x+y)^{l}Y_{W}(Y(u,x)v,y)w.
\end{eqnarray}

The notion of associate of the formal group $F_{\rm a}$ was designed in \cite{L6}
to be an analog of that of $G$-set of a group $G$.
By definition, an associate of $F_{\rm a}$ is a formal series
$\phi(x,z)\in \bC((x))[[z]]$ such that
$$\phi(x,0)=x\ \ \mbox{ and }\ \ \ \phi(x,\phi(y,z))=\phi(x,y+z).$$
Interestingly, it was proved therein that for any $p(x)\in \bC((x))$, $\phi_{p(x)}(x,z):=e^{zp(x)d/dx}x$ is an associate of $F_{\rm a}$
and every associate of $F_{\rm a}$ is of this form. When $p(x)=1$, we get the formal group itself, whereas
when $p(x)=x$, we get $\phi_{p(x)}(x,z)=xe^{z}$.

Let $\phi(x,z)$ be a general associate of $F_{\rm a}$.
The notion of $\phi$-coordinated (quasi) $V$-module for a vertex algebra (more generally for a weak quantum vertex algebra in the sense of \cite{L2}) $V$ was defined by replacing the ordinary weak associativity axiom
with the property that for any $u,v\in V$, there is a nonnegative integer $k$ such that
$$(x_{1}-x_{2})^{k}Y_{W}(u,x_{1})Y_{W}(v,x_{2})\in \Hom (W,W((x_{1},x_{2})))$$ and
\begin{eqnarray}
\left((x_{1}-x_{2})^{k}Y_{W}(u,x_{1})Y_{W}(v,x_{2})\right)|_{x_{1}=\phi(x_{2},z)}
=(\phi(x_{2},z)-x_{2})^{k}Y_{W}(Y(u,z)v,x_{2}).
\end{eqnarray}
When $\phi(x,z)=F_{\rm a}(x,z)=x+z$, this gives an equivalent definition of the ordinary notion of module
(cf. \cite{LTW}).
In \cite{L6}, the main focus is on $\phi$-coordinated (quasi) modules with $\phi(x,z)=e^{zx\frac{d}{dx}}x=xe^{z}$, by which weak quantum vertex algebras
were canonically associated to quantum affine algebras.
Among the main results, a Jacobi-type identity and a commutator formula for
$\phi$-coordinated modules were obtained.
Later in \cite{L10}, $\phi$-coordinated quasi modules were studied furthermore, where
a commutator formula, similar to that for twisted modules (see \cite{FLM}), was obtained.
Just as commutator formulas for modules and twisted modules are very important and useful
in the vertex algebra theory, such commutator formulas for $\phi$-coordinated (quasi) modules were proved to be very useful.

In this current paper, we study $\phi_{\ep}$-coordinated modules for vertex algebras,
where $\phi_{\ep}(x,z)=e^{zx^{\ep}\frac{d}{dx}}x$ with $\ep$ an {\em arbitrary} integer.
Part of our motivation is the fact
that $x^{n}\frac{d}{dx}$ with $n\in \bZ$ form a basis of the Witt algebra,
which plays a vital role in vertex operator algebra theory and in physics conformal field theory.
Conceivably, $\phi_{\ep}$-coordinated modules for vertex algebras will be of fundamental
importance.
Among the main results, we show that if $U$ is a local subset of $\Hom (W,W((x)))$ with $W$ a general vector space, the nonlocal vertex algebra $\lan U\ran_{\phi_{\ep}}$ that was obtained in \cite{L6} is a vertex algebra.
We also obtain a Jacobi-type identity for $\phi_{\ep}$-coordinated modules for vertex algebras and furthermore we derive a commutator formula.

In this paper, we also study $\phi_{\ep}$-coordinated modules for some special family of vertex algebras.
Note that for any Lie algebra $\mg$ equipped with a symmetric invariant bilinear form $\lan\cdot,\cdot\ran$,
one has an (untwisted) affine Lie algebra $\hat{\mg}=\mg\otimes \bC[t,t^{-1}]+\bC \c$. Furthermore, for every complex number $\ell$, one has a vertex algebra $V_{\hat{\mg}}(\ell,0)$.
This gives a very important family of vertex algebras.
There is another family of vertex algebras which are associated to Novikov algebras.
Recall that a (left) Novikov algebra is a non-associative algebra $\A$ satisfying the condition that
\begin{eqnarray*}
(ab)c-a(bc)=(ba)c-b(ac),\ \ \ \
(ab)c=(ac)b\ \ \mbox{ for }a,b,c\in \A.
\end{eqnarray*}
A result of Primc (see \cite{Pr}) is that for any given (left) Novikov algebra $\A$ equipped with a symmetric bilinear form $\lan\cdot,\cdot\ran$ such that
$$\lan ab,c\ran=\lan a,bc\ran,\ \ \ \lan ab,c\ran=\lan ba,c\ran\ \ \mbox{ for }a,b,c\in \A,$$
one has a Lie algebra $\widetilde{L}(\A)=\A\otimes \bC[t,t^{-1}]+\bC \c$, and furthermore, for every complex number $\ell$ one has a vertex algebra $V_{\widetilde{L}(\A)}(\ell,0)$.

In this current paper, as an application of our general results we study $\phi_{\ep}$-coordinated modules for vertex algebras $V_{\widetilde{L}(\A)}(\ell,0)$ associated to Novikov algebras $\A$. To determine $\phi_{\ep}$-coordinated modules we introduce a Lie algebra $\widetilde{L}^{\ep}(\A)$, which has the same underlying space as that of $\widetilde{L}(\A)$. We show that a $\phi_{\ep}$-coordinated module structure on a vector space for vertex algebra $V_{\widetilde{L}(\A)}(\ell,0)$
exactly amounts to a ``restricted'' module structure for the Lie algebra $\widetilde{L}^{\ep}(\A)$ of level $\ell$.

Throughout this paper, $\bZ$ and $\bN$ denote the set of nonnegative integers
and the set of integers, respectively, and all vector spaces are assumed to be over the
field $\bC$ of complex numbers. On the other hand,
we shall use the standard formal variable notations and conventions (see \cite{FLM}, \cite{FHL}, and \cite{LL}).

This paper is organized as follows: In Section 2, we recall the basic results on associates and we derive an explicit formula for $\phi_{\ep}$. In Section 3, we prove a Jacobi-type identity and a commutator formula for $\phi_{\ep}$-coordinated modules.
In Section 4, we study $\phi_{\ep}$-coordinated modules for the vertex algebras associated to Novikov algebras.

%%%%%%%%%%%%%%%%%%%%%%%%
\section{One-dimensional additive formal group and $\phi_\ep$-coordinated modules for vertex algebras}
%%%%%%%%%%%%%%%%%%%%%%%%
\subsection{}
We here briefly recall the notion of formal group (cf. \cite{Ha}) and the
notion of associate of a formal group (see \cite{L6}).

\begin{defi}\label{FG}
{\em A {\em one-dimensional formal group} over $\bC$ is a formal power series
$F(x, y)\in \bC[[x, y]]$ such that
$$
F(x, 0) = x,\quad F(0, y) = y,\quad F(x, F(y, z))= F(F(x, y), z).
$$}
\end{defi}

The simplest example is the one-dimensional additive formal group
\begin{eqnarray}
F_{\rm a}(x, y) = x + y.
\end{eqnarray}

The following notion of associate of a formal group was
introduced in \cite{L6}:

\begin{defi}\label{ASSO}
{\em Let $F(x, y)$ be a one-dimensional formal group over $\bC$. An {\em associate} of
$F(x, y)$ is a formal series $\phi(x, z) \in \bC((x))[[z]]$, satisfying the condition that
$$
\phi(x, 0) = x,\quad \phi(\phi(x, x_0), x_2) = \phi(x, F(x_0, x_2)).
$$}
\end{defi}

The following is an explicit classification of associates for $F_{\rm a}(x, y)$ obtained in \cite{L6}:

\begin{pro}
Let $p(x)\in\bC((x))$. Set
$$\phi(x,z)=e^{z(p(x)d/dx)}x=\sum_{n\geq0}\frac{z^n}{n!}\left(p(x)\frac{d}{d x}\right)^nx\in \bC((x))[[z]].$$
Then $\phi(x, z)$ is an associate of $F_{\rm a}(x, y)$. Furthermore, every associate of
$F_{\rm a}(x, y)$ is of this form with $p(x)$ uniquely determined.
\end{pro}

For any integer $\ep$, set
\begin{eqnarray}
\phi_\ep(x,z)= e^{z(x^{\ep} d/dx)}x,
\end{eqnarray}
an associate of $F_{\rm a}(x, y)$.
For the rest of this paper, we shall be only concerned about the
$1$-dimensional additive formal group $F_{\rm a}(x, y)$ and its associates $\phi_\ep(x,z)$.

As special cases, we have (see \cite{L6})
\begin{eqnarray}
&&\phi_{0}(x,z)= x+z,\quad  \phi_1(x,z)= xe^z,\quad \phi_2(x,z)= \frac{x}{1-zx}.
\end{eqnarray}
For the general case, from definition we have
\begin{eqnarray}
\phi_{\ep+1}(x,z)&=&x+\sum_{k\geq 1} \frac{z^k}{k!}x^{k\ep+1}\prod_{j=0}^{k-1}(1+j\ep).
%&=&\begin{cases}&xe^z,\quad\quad\quad\quad\quad\ \text{for}\ \ep=0,\\
%&x(1-\ep zx^\ep)^{-\frac{1}{\ep}},\quad\text{for}\ \ep\neq0.
%\end{cases}
\end{eqnarray}
Assume $\ep\neq 0$.  For $k\ge 1$, we have
\begin{eqnarray*}
\frac{1}{k!}\prod_{j=0}^{k-1}(1+j\ep)=(-\ep)^{k}\frac{1}{k!}\prod_{j=0}^{k-1}(-\frac{1}{\ep}-j)
=(-\ep)^{k}\binom{-\frac{1}{\ep}}{k}.
\end{eqnarray*}
Then
\begin{eqnarray}
\phi_{\ep+1}(x,z)=x+x\sum_{k\ge 1}(-\ep)^{k}\binom{-\frac{1}{\ep}}{k}(zx^{\ep})^{k}
=x(1-\ep zx^\ep)^{-\frac{1}{\ep}}.
\end{eqnarray}
Notice that
\begin{eqnarray}
\lim_{\ep\to 0}\phi_{\ep+1}(x,z)=xe^z=\phi_1(x,z).
\end{eqnarray}

The following are some simple facts we shall use:

\begin{lem}\label{esimple-fact}
 We have
\begin{eqnarray}
 \phi_{\ep}(x,z)-x&=&zh(x,z),\\
\phi_{\ep}(x,x_{1})-\phi_{\ep}(x,x_{2})&=&(x_{1}-x_{2})g(x,x_{1},x_{2}),
\end{eqnarray}
where $h(x,z)$ is a unit in $\bC((x))[[z]]$ and
$g(x,x_{1},x_{2})$ is a unit in $\bC((x))[[x_{1},x_{2}]]$.
\end{lem}

\begin{proof} By definition we have
$$\phi_{\ep}(x,z)-x=z\sum_{j\ge 1}\frac{1}{j!}z^{j-1}\left(x^{\ep}\frac{d}{dx}\right)^{j}x.$$
Set $h(x,z)=\sum_{j\ge 1}\frac{1}{j!}z^{j-1}\left(x^{\ep}\frac{d}{dx}\right)^{j}x\in \bC((x))[[z]]$.
As $h(x,0)=1$, $h(x,z)$ is a unit in $\bC((x))[[z]]$.
On the other hand, by definition we have
\begin{eqnarray*}
&&\phi_{\ep}(x,x_{1})-\phi_{\ep}(x,x_{2})
=\left(e^{x_{1}x^{\ep}\frac{d}{dx}}-e^{x_{2}x^{\ep}\frac{d}{dx}}\right)x
=\left(e^{(x_{1}-x_{2})x^{\ep}\frac{d}{dx}}-1\right)e^{x_{2}x^{\ep}\frac{d}{dx}}\cdot x\\
&&=(x_{1}-x_{2})\sum_{j\ge 1}\frac{1}{j!}(x_{1}-x_{2})^{j-1}
\left(x^{\ep}\frac{d}{dx}\right)^{j}\phi_{\ep}(x,x_{2}).
\end{eqnarray*}
Set
$$g(x,x_{1},x_{2})=\sum_{j\ge 1}\frac{1}{j!}(x_{1}-x_{2})^{j-1}
\left(x^{\ep}\frac{d}{dx}\right)^{j}\phi_{\ep}(x,x_{2}).$$
We have $g(x,x_{1},x_{2})\in \bC((x))[[x_{1},x_{2}]]$ and $g(x,x_{2},x_{2})=\phi_{\ep}(x,x_{2})$.
Note that $\phi_{\ep}(x,x_{2})$ is a unit in $\bC((x))[[x_{2}]]$
as $\phi_{\ep}(x,0)=x$ (nonzero in $\bC((x))$).
It then follows that $g(x,x_{1},x_{2})$ is a unit in $(\bC((x))[[x_{2}]])[[x_{1}]]=\bC((x))[[x_{1},x_{2}]]$.
\end{proof}

\subsection{}
Next we recall the definition of a (nonlocal) vertex algebra
and the definitions of a module and a $\phi_\ep$-coordinated module for a (nonlocal) vertex
algebra. The notion of nonlocal vertex algebra was studied in \cite{Lg1} (under the name ``axiomatic $G_{1}$-vertex algebra'') and in \cite{L2}, and it was also independently studied in \cite{BK} (under the name ``field algebra'').
The theory of $\phi$-coordinated modules for quantum vertex algebras
was developed in \cite{L7}, whereas $\phi_{1}$-coordinated modules
were the main focus therein.

\begin{defi}\label{defnonlocalva}
{\em A {\em nonlocal vertex algebra} is a vector space $V$ equipped with a linear map
\begin{eqnarray*}
Y(\cdot,x):&&\ V\rightarrow \Hom (V,V((x)))\subset (\End V)[[x,x^{-1}]]\\
&& v\mapsto Y(v,x)=\sum_{\bZ}v_{n}x^{-n-1}\ \ (\mbox{where }v_{n}\in \End V)
\end{eqnarray*}
and with a distinguished vector ${\bf 1}\in V$, satisfying the conditions that
$$Y({\bf 1},x)=1,\ \ \ Y(v,x){\bf 1}\in V[[x]]\ \mbox{ and }
\ \lim_{x\rightarrow 0}Y(v,x){\bf 1}=v\  \mbox{ for }v\in V$$
and that for any $u,v,w\in V$, there exists a nonnegative integer $l$ such that
\begin{eqnarray}
(x_{0}+x_{2})^{l}Y(u,x_{0}+x_{2})Y(v,x_{2})w=(x_{0}+x_{2})^{l}Y(Y(u,x_{0})v,x_{2})w.
\end{eqnarray}
Furthermore, a {\em vertex algebra} is a nonlocal vertex algebra $V$ satisfying
the condition that for any $u,v\in V$,
there exists a nonnegative integer $k$ such that
\begin{eqnarray}
(x_{1}-x_{2})^{k}Y(u,x_{1})Y(v,x_{2})=(x_{1}-x_{2})^{k}Y(v,x_{2})Y(u,x_{1}).
\end{eqnarray}}
\end{defi}

For each nonlocal vertex algebra $V$, there is a canonical operator $\D$ on $V$, which is defined by
$$\D(v)=v_{-2}{\bf 1}=\left(\frac{d}{dx}Y(v,x){\bf 1}\right)|_{x=0}
\ \ \ \mbox{ for }v\in V.$$
This operator $\D$ satisfies the following property:
\begin{eqnarray}
[\D, Y(v,x)]=Y(\D v,x)=\frac{d}{dx}Y(v,x).
\end{eqnarray}

\begin{defi}\label{module}
{\em Let $V$ be a vertex algebra.  A $V$-{\em module} is a vector space $W$ equipped with a linear map
\begin{eqnarray*}
Y_W(\cdot, x) :&& V \to  \Hom(W,W((x))) \subset (\End W)[[x, x^{-1}]],\\
&&v\mapsto Y_W(v, x),
\end{eqnarray*}
satisfying the conditions that $Y_W(1, x) = 1_W$ (the identity operator on $W$) and that for
$u, v \in V,\ w \in W$, there exists $l\in\bN$ such that
\begin{eqnarray}
(x_0+x_2)^lY_W(u,x_0+x_2)Y_W(v,x_2)=(x_0+x_2)^lY_W(Y(u,x_0)v,x_2)w.
\end{eqnarray}}
\end{defi}

\begin{rem}{\rm It was shown in \cite{LTW} (Lemma 2.9) that the weak associativity axiom in
the definition of a $V$-module can be equivalently replaced by the condition that for any
$u, v \in V$, there exists $k \in \bN$ such that
\begin{eqnarray}
&&(x_1-x_2)^kY_W(u, x_1)Y_W(v, x_2)\in\Hom(W,W((x_1, x_2))),\\
&& x_0^kY_W(Y(u,x_0)v,x_2)=\left((x_1-x_2)^kY_W(u,x_1)Y_W(v,x_2)\right)|_{x_1=x_2+x_0}.
\end{eqnarray}}
\end{rem}

\begin{defi}\label{dcq}
{\em Let $V$ be a nonlocal vertex algebra and let $\phi$ be an associate of the one-dimensional additive formal group $F_{\rm a}$.  A {\em $\phi$-coordinated $V$-module} is a vector space $W$ equipped with a linear map $Y_W(\cdot, x)$ as in Definition \ref{module},
satisfying the conditions that $Y_W(1, x) = 1_W$ and that for $u, v
\in V$, there exists $k \in \bN$ such that
\begin{eqnarray}
&&(x_1-x_2)^kY_W(u,x_1)Y_W(v,x_2)\in\Hom(W,W((x_1,x_2))),\\
&& (\phi(x_2,x_0)-x_2)^kY_W(Y(u,x_0)v,x_2)\nonumber\\
&=&((x_1-x_2)^kY_W(u,x_1)Y_W(v,x_2))|_{x_1=\phi(x_2,x_0)}.
\end{eqnarray}}
\end{defi}

It is clear that the notion of {$\phi_{0}$-coordinated $V$-module}
 is equivalent to that of $V$-module.

%%%%%%%%%%%%%%%%%%%%%%%%
\section{Jacobi-type identity for $\phi_\ep$-coordinated modules}
%%%%%%%%%%%%%%%%%%%%%%%%

In this section, we shall present some axiomatic results on $\phi_\ep$-coordinated modules for vertex algebras.
In particular, we establish a Jacobi-type identity and a commutator formula.

Let $W$ be a vector space. Set
\begin{eqnarray}
\E(W)=\Hom (W,W((x)))\subset (\End W)[[x,x^{-1}]].
\end{eqnarray}
A subset $U$ of $\E(W)$ is said to be {\em local} if for any $a(x),b(x)\in U$ there exists a nonnegative integer $k$  such that
\begin{eqnarray}\label{elocality-def}
(x_1-x_2)^ka(x_1)b(x_2)=(x_1-x_2)^kb(x_2)a(x_1).
\end{eqnarray}
A pair $(a(x), b(x))$ in $\E(W)$ is said
to be {\em compatible} if there exists $k \in \bN$ such that
\begin{eqnarray}\label{ecompatible-condition}
(x_1-x_2)^ka(x_1)b(x_2)\in \Hom (W, W((x_1,x_2))).
\end{eqnarray}
Note that (\ref{elocality-def}) implies (\ref{ecompatible-condition}).
Thus each pair in a local subset is always compatible.

Fix an integer $\ep$ throughout this section.
Let $(a(x), b(x))$ be any compatible pair in $\E(W)$ with $k \in
\bN$ such that (\ref{ecompatible-condition}) holds. We define $a(x)_{n}^{\ep}b(x)\in \E(W)$ for $n\in \bZ$
in terms of generating function
$$Y_{\E}^{\ep} (a(x),z)b(x)=\sum_{n\in \bZ}a(x)_{n}^{\ep}b(x)z^{-n-1}$$
by
\begin{eqnarray}
Y_{\E}^{\ep} (a(x),z)b(x)=(\phi_\ep(x,z)-x)^{-k}\left((x_1-x)^ka(x_1)b(x)\right)|_{x_1=\phi_\ep(x,z)},
\end{eqnarray}
where $(\phi_\ep(x,z)-x)^{-k}$ is viewed as an element of $\bC((x))((z))$
(recalling Lemma \ref{esimple-fact}).

\begin{rem} {\em A notion of compatible subset of $\E(W)$ was introduced in \cite{Lg1} and it was proved therein that each compatible subset generates a nonlocal vertex algebra in a certain canonical way.
It was proved in \cite{L6} (Theorem 4.10) that any compatible subset $U$ of $\E(W)$ generates a nonlocal vertex algebra $\lan U\ran_{\phi_{\ep}}$ with $W$ as a canonical $\phi_{\ep}$-coordinated module.
On the other hand, it was proved in \cite{Lg1} that every local subset of $\E(W)$ is compatible.
Thus each local subset $U$ of $\E(W)$ generates a nonlocal vertex algebra $\lan
U\ran_{\phi_{\ep}}$.
It was proved in \cite{L6} that $\lan U\ran_{\phi_{1}}$ is a vertex
algebra, whereas it was proved in \cite{L1} that $\lan U\ran_{\phi_{0}}$ is a vertex algebra.}
\end{rem}

In the following, we shall prove that for every integer $\ep$, $\lan U\ran_{\phi_{\ep}}$ is a
vertex algebra, generalizing the corresponding results of \cite{L1, L6}.

\begin{pro}\label{P}
Let $W$ be a vector space and let $V$ be a local subspace of $\E(W)$,
which is $Y^{\ep}_{\E}$-closed in the sense that
$$u(x)_{n}^{\ep}v(x)\in V\ \mbox{ for }u(x),v(x)\in V,\ n\in \bZ.$$
Let $a(x), b(x)\in V$. Suppose
\begin{eqnarray}\label{T}
(x_1-x_2)^ka(x_1)b(x_2)=(x_1-x_2)^kb(x_2)a(x_1)
\end{eqnarray}
for some nonnegative integer $k$.
Then
\begin{eqnarray}
(x_1-x_2)^kY_{\E}^{\ep}(a(x),x_1)Y_{\E}^{\ep}(b(x),x_2)
=(x_1-x_2)^kY_{\E}^{\ep}(b(x),x_2)Y_{\E}^{\ep}(a(x),x_1).
\end{eqnarray}
\end{pro}

\begin{proof}
Let $c(x)\in V$ be arbitrarily fixed. There exists $l\in \bN$ with $l\ge k$ such that
\begin{eqnarray*}
(z-x)^la(z)c(x)=(z-x)^{l}c(x)a(z),\ \ \ \ (z-x)^lb(z)c(x)=(z-x)^{l}c(x)b(z).
\end{eqnarray*}
Using this and (\ref{T}) we get
\begin{eqnarray*}
(y-z)^l(y-x)^l(z-x)^l a(y)b(z)c(x)\in\Hom(W,W((x,y,z))).
\end{eqnarray*}
By Lemma 4.7 in \cite{L6}, we have
\begin{eqnarray*}
&&(\phi_\ep(x,x_1)-\phi_\ep(x,x_2))^l(\phi_\ep(x,x_1)-x)^l(\phi_\ep(x,x_2)-x)^lY_\E^{\ep}(a(x),x_1)Y_\E^{\ep}(b(x),x_2)c(x)\\
&&=(y-z)^l(y-x)^l(z-x)^l a(y)b(z)c(x)|_{y=\phi_\ep(x,x_1), z=\phi_\ep(x,x_2)}.
\end{eqnarray*}
Set
\begin{eqnarray*}
&&f(x,x_1,x_2)=(\phi_\ep(x,x_1)-\phi_\ep(x,x_2))^l(\phi_\ep(x,x_1)-x)^l(\phi_\ep(x,x_2)-x)^l,
\end{eqnarray*}
which lies in $\bC((x))((x_{1},x_{2}))$.
Then
\begin{eqnarray*}
&&f(x,x_1,x_2)({\phi_\ep(x,x_1)}-{\phi_\ep(x,x_2)})^kY_\E^{\ep}(a(x),x_1)Y_\E^{\ep}(b(x),x_2)c(x)\\
&=&(y-z)^l(y-x)^l(z-x)^l(y-z)^k a(y)b(z)c(x)|_{y=\phi_\ep(x,x_1), z=\phi_\ep(x,x_2)}\\
&=&(y-z)^l(y-x)^l(z-x)^l(y-z)^kb(z)a(y)c(x)|_{z=\phi_\ep(x,x_2),y=\phi_\ep(x,x_1)}\\
&=&f(x,x_1,x_2)({\phi_\ep(x,x_1)}-{\phi_\ep(x,x_2)})^kY_\E^{\ep}(b(x),x_2)Y_\E^{\ep}(a(x),x_1)c(x).
\end{eqnarray*}
Noticing that $(\phi_\ep(x,x_1)-x)^l(\phi_\ep(x,x_2)-x)^l$ is invertible in $\bC((x))((x_{1},x_{2}))$, by cancellation, we get
\begin{eqnarray}
&&(\phi_\ep(x,x_1)-\phi_\ep(x,x_2))^{l+k}Y_\E^{\ep}(a(x),x_1)Y_\E^{\ep}(b(x),x_2)c(x)\nonumber\\
&&=(\phi_\ep(x,x_1)-\phi_\ep(x,x_2))^{l+k}Y_\E^{\ep}(b(x),x_2)Y_\E^{\ep}(a(x),x_1)c(x).
\end{eqnarray}
By Lemma \ref{esimple-fact}, we have
\begin{eqnarray*}
(\phi_\ep(x,x_1)-\phi_\ep(x,x_2))^{l+k}=(x_1-x_2)^{l+k}g(x,x_1,x_2)^{l+k},
\end{eqnarray*}
where $g(x,x_1,x_2)$ is a unit in $\bC((x))[[x_1, x_2]]$. By cancellation we get
\begin{eqnarray}
(x_1-x_2)^{l+k}Y_{\E}^{\ep}(a(x),x_1)Y_{\E}^{\ep}(b(x),x_2)
=(x_1-x_2)^{l+k}Y_{\E}^{\ep}(b(x),x_2)Y_{\E}^{\ep}(a(x),x_1).
\end{eqnarray}
Combining this with the weak associativity obtained in \cite{L6}  we get
\begin{eqnarray}\label{JJ}
&&x_0^{-1}\de\left(\frac{x_1-x_2}{x_0}\right)Y_{\E}^{\ep}(a(x),x_1)Y_{\E}^{\ep}(b(x),x_2)
-x_0^{-1}\de\left(\frac{x_2-x_1}{-x_0}\right)Y_{\E}^{\ep}(b(x),x_2)Y_{\E}^{\ep}(a(x),x_1)\nonumber\\
&&=x_2^{-1}\de\left(\frac{x_1-x_0}{x_2}\right)Y_{\E}^{\ep}(Y_{\E}^{\ep}(a(x),x_0)b(x),x_2).
\end{eqnarray}
From (\ref{T}) we have
$$
(x_1-x_2)^ka(x_1)b(x_2)\in \Hom (W, W((x_1,x_2))),$$
so that
$$
({\phi_\ep(x_2,x_0)}-x_2)^kY_{\E}^{\ep}(a(x),x_0)b(x)
=(x_1-x)^ka(x_1)b(x)|_{x_1=\phi_\ep(x_2,x_0)}.
$$
Multiplying both sides of (\ref{JJ})
by $({\phi_\ep(x_2,x_0)}-x_2)^k$ and then taking $\Res_{x_0}$ we get
\begin{eqnarray*}
&&({\phi_\ep(x_2,x_0)}-x_2)^kY_{\E}^{\ep}(a(x),x_1)Y_{\E}^{\ep}(b(x),x_2)\\
&=&({\phi_\ep(x_2,x_0)}-x_2)^kY_{\E}^{\ep}(b(x),x_2)Y_{\E}^{\ep}(a(x),x_1).
\end{eqnarray*}
By Lemma \ref{esimple-fact}, we have
$$({\phi_\ep(x_2,x_0)}-x_2)^k=x_0^kh(x_2,x_0)^{k}$$
where $h(x_2,x_0)$ is a unit in $\bC((x_2))[[x_0]]$. By cancellation we obtain
\begin{eqnarray*}
(x_1-x_2)^kY_{\E}^{\ep}(a(x),x_1)Y_{\E}^{\ep}(b(x),x_2)
=(x_1-x_2)^kY_{\E}^{\ep}(b(x),x_2)Y_{\E}^{\ep}(a(x),x_1),
\end{eqnarray*}
as desired.
\end{proof}

Now we have:

\begin{theo}\label{GT}
Let $W$ be a vector space and let $U$ be any local subset
of $\E(W)$. Then  $\lan U\ran_{\phi_{\ep}}$ is a vertex algebra and $W$ is a canonical $\phi_\ep$-coordinated $\lan U\ran_{\phi_{\ep}}$-module.
\end{theo}

\begin{proof} We already knew that $\lan U\ran_{\phi_{\ep}}$ is a nonlocal vertex algebra and
 $W$ is a $\phi_\ep$-coordinated $\lan U\ran_{\phi_{\ep}}$-module
with $Y_{W}(\alpha(x),z)=\alpha(z)$ for $\alpha(x)\in \lan U\ran_{\phi_{\ep}}$.
As $\lan U\ran_{\phi_{\ep}}$ is the smallest $Y^\ep_\E$-closed local subspace containing $U$ and $1_W$, we see that
$\lan U\ran_{\phi_{\ep}}$ as a nonlocal vertex algebra is generated by $U$. Given that $U$ is local,
by Proposition \ref{P} we see that
$$
\{Y^\ep_\E(a(x), z)\mid a(x) \in U\}
$$
is a local subset of $\E(\lan U\ran_{\phi_{\ep}})$.
It follows that $\lan U\ran_{\phi_{\ep}}$ is a vertex algebra and
$W$ is a $\phi_\ep$-coordinated $\lan U\ran_{\phi_{\ep}}$-module.
\end{proof}

We also have the following results:

\begin{pro}\label{P2}
Let $V$ be a vertex algebra and let $(W, Y_W)$ be a $\phi_\ep$-coordinated $V$-module.
Suppose that for some fixed $u,v\in V,\ k\in \bN$,
$$
(x_1-x_2)^kY_W(u, x_1)Y_W(v, x_2)\in\Hom(W,W((x_1, x_2))).
$$
Then
\begin{eqnarray}\label{elocal-prop}
(x_1-x_2)^kY_W(u,x_1)Y_W(v,x_2)=(x_1-x_2)^kY_W(v,x_2)Y_W(u,x_1).
\end{eqnarray}
\end{pro}

\begin{proof}
Recall the skew-symmetry of  $V$:
\begin{eqnarray*}
Y(u,x)v=e^{x\D}Y(v,-x)u\ \ \ \mbox{ for }u,v\in V.
\end{eqnarray*}
From the definition, there exists  $l\in \bN$ such that
$$
(x_1-x_2)^lY_W(u, x_1)Y_W(v, x_2),\quad (x_1-x_2)^lY_W(v, x_2)Y_W(u, x_1)\in\Hom(W,W((x_1, x_2))).
$$
Then, using Lemmas 3.6 and 3.7 in \cite{L6} we get
\begin{eqnarray*}
&&((x_1-x_2)^lY_W(u, x_1)Y_W(v, x_2))|_{x_1=\phi_\ep(x_2,x_0)}\\
&=&(\phi_\ep(x_2,x_0)-x_2)^lY_W(Y(u, x_0)v, x_2)\\
&=&(\phi_\ep(x_2,x_0)-x_2)^lY_W(e^{x_0\D}Y(v, -x_0)u, x_2)\\
&=&(\phi_\ep(x_2,x_0)-x_2)^lY_W(Y(v, -x_0)u, \phi_\ep(x_2,x_0)).
\end{eqnarray*}
On the other hand, we have
\begin{eqnarray*}
&&((x_1-x_2)^lY_W(v, x_2)Y_W(u, x_1))|_{x_2=\phi_\ep(x_1,-x_0)}=(x_1 -\phi_\ep(x_1,-x_0))^lY_W(Y(v, -x_0)u, x_1).
\end{eqnarray*}
Hence
\begin{eqnarray*}
&&((x_1-x_2)^lY_W(u, x_1)Y_W(v, x_2))|_{x_1=\phi_\ep(x_2,x_0)}\\
&=&(((x_1-x_2)^lY_W(v, x_2)Y_W(u, x_1))|_{x_2=\phi_\ep(x_1,-x_0)})|_{x_1=\phi_\ep(x_2,x_0)}\\
&=&((x_1-x_2)^lY_W(v, x_2)Y_W(u, x_1))|_{x_1=\phi_\ep(x_2,x_0)}.
\end{eqnarray*}
It follows that
\begin{eqnarray*}
(x_1-x_2)^lY_W(u, x_1)Y_W(v, x_2)=(x_1-x_2)^lY_W(v, x_2)Y_W(u, x_1),
\end{eqnarray*}
which implies
\begin{eqnarray*}
(x_1-x_2)^l(x_1-x_2)^kY_W(u, x_1)Y_W(v, x_2)=(x_1-x_2)^l(x_1-x_2)^kY_W(v, x_2)Y_W(u, x_1).
\end{eqnarray*}
Noticing that $(x_1-x_2)^kY_W(v, x_2)Y_W(u, x_1)\in \Hom (W,W((x_{2}))((x_{1})))$  and
$$
(x_1-x_2)^kY_W(u, x_1)Y_W(v, x_2)\in\Hom(W,W((x_1, x_2)))\subset \Hom(W,W((x_{2}))((x_{1})))
$$
by assumption, we can multiply both sides by the inverse of $(x_1-x_2)^l$ in $\bC((x_2))((x_1))$ to obtain the desired relation (\ref{elocal-prop}).
\end{proof}

\begin{pro}\label{P4}
Let $V$ be a vertex algebra and let $(W, Y_W)$
be a $\phi_\ep$-coordinated  $V$-module. Set $V_W = \{Y_W(v, x) | v \in V \}$. Then  $V_W$ is a
local subspace of $\E(W)$, $(V_W, Y_{\E}^{\ep} , 1_W)$ is a vertex algebra,
and $Y_W$ is a homomorphism of  vertex algebras.
\end{pro}

\begin{proof}
By Definition \ref{dcq} and Proposition \ref{P2},  $V_W$ is a local subspace of $\E(W)$. Let $u, v \in V$.
By definition, there exists $k\in\bN$ such that
$$(x_1-x)^kY_W(u, x_1)Y_W(v, x)\in \Hom(W,W((x_1, x)))$$
and
$$
\left({\phi_\ep(x_2,z)-x_2}\right)^kY_{\E}^\ep(Y(u,x),z)Y_W(v,x)
=(x_1-x)^kY_W(u,x_1)Y_W(v,x)|_{x_1=\phi_\ep(x_2,z)}.
$$
On the other hand, from the definition of $Y^\ep_\E (\cdot, x)$ we have
$$
\left({\phi_\ep(x_2,z)-x_2}\right)^kY_{\E}^\ep(Y_W(u,x),z)Y_W(v,x)
=(x_1-x)^kY_W(u,x_1)Y_W(v,x)|_{x_1=\phi_\ep(x_2,z)}.
$$
It follows that
$$
\left({\phi_\ep(x_2,z)-x_2}\right)^kY_{\E}^\ep(Y(u,x),z)Y_W(v,x)=\left({\phi_\ep(x_2,z)-x_2}\right)^kY_{\E}^\ep(Y_W(u,x),z)Y_W(v,x).
$$
Since both $Y_{\E}^\ep(Y(u,x),z)Y_W(v,x)$ and $Y_{\E}^\ep(Y_W(u,x),z)Y_W(v,x)$ involve only finitely many
negative powers of $z$, and since $(\phi_{\ep}(x_2,z)-x_2)^k$ is a unit in $\bC((x_{2}))((z))$ (by Lemma \ref{esimple-fact}),
by cancellation we get
$$
Y_{\E}^\ep(Y(u,x),z)Y_W(v,x)=Y_{\E}^\ep(Y_W(u,x),z)Y_W(v,x).
$$
Then $(V_W, Y_{\E}^{\ep} , 1_W)$ is a vertex algebra,
and $Y_W$ is a homomorphism of  vertex algebras.
\end{proof}

We also have the following result generalizing the corresponding results of \cite{Lg1, L6}:

\begin{lem}\label{Le}
 Let $W$ be a vector space and let
\begin{eqnarray*}
&&A(x_1,x_2)\in \Hom(W,W((x_1))((x_2))),\quad B(x_1, x_2) \in \Hom(W,W((x_2))((x_1))),\\
&&C(x_0, x_2)\in (\Hom(W,W((x_2))))((x_0)).
\end{eqnarray*}
If there exists a nonnegative integer $k$ such that
\begin{eqnarray*}
&&(x_1-x_2)^k A(x_1,x_2)= (x_1-x_2)^kB(x_1, x_2),\\
&&\left((x_1-x_2)^k A(x_1, x_2)\right)|_{x_1=\phi_\ep(x_2,x_0)}=(\phi_\ep(x_2,x_0)-x_2)^kC(x_0,x_2),
\end{eqnarray*}
then
\begin{eqnarray}\label{L}
&&(x_2z)^{-1}\de\left(\frac{x_1-x_2}{x_2z}\right)A(x_1,x_2)
-(x_2z)^{-1}\de\left(\frac{x_2-x_1}{-x_2z}\right)B(x_1,x_2)\nonumber\\
&&=x_1^{-1}\de\left(\frac{x_2(1+z)}{x_1}\right)C(f_\ep(x_2,z),x_2),
\end{eqnarray}
where
$$
f_\ep(x_2,z)=\begin{cases}
&x_2^{1-\ep}\cdot \frac{(1+z)^{1-\ep}-1}{1-\ep},\quad \ \ \text{for}\ \ep\neq1,\\
&\log (1+z),\quad\quad\quad\quad \text{for}\ \ep=1.
\end{cases}
$$
\end{lem}

\begin{proof}
We start with the standard delta-function identity
\begin{eqnarray*}
&&x_0^{-1}\de\left(\frac{x_1-x_2}{x_0}\right)-x_0^{-1}\de\left(\frac{x_2-x_1}{-x_0}\right)=x_2^{-1}\de\left(\frac{x_1-x_0}{x_2}\right).
\end{eqnarray*}
Substituting $x_0 = x_2z$ with $z$ a new formal variable, we have
\begin{eqnarray}
&&(x_2z)^{-1}\de\left(\frac{x_1-x_2}{x_2z}\right)
-(x_2z)^{-1}\de\left(\frac{x_2-x_1}{-x_2z}\right)=x_1^{-1}\de\left(\frac{x_2(1+z)}{x_1}\right).
\end{eqnarray}
Then we get
\begin{eqnarray*}
&&(x_2z)^{-1}\de\left(\frac{x_1-x_2}{x_2z}\right)(x_2z)^kA(x_1,x_2)
-(x_2z)^{-1}\de\left(\frac{x_2-x_1}{-x_2z}\right)(x_2z)^kB(x_1,x_2)\nonumber\\
&=&(x_2z)^{-1}\de\left(\frac{x_1-x_2}{x_2z}\right)(x_1-x_2)^kA(x_1,x_2)
-(x_2z)^{-1}\de\left(\frac{x_2-x_1}{-x_2z}\right)(x_1-x_2)^kB(x_1,x_2)\\
&=&x_1^{-1}\de\left(\frac{x_2(1+z)}{x_1}\right)\left((x_1-x_2)^kA(x_1,x_2)\right)\\
&=&x_1^{-1}\de\left(\frac{x_2(1+z)}{x_1}\right)\left((x_1-x_2)^kA(x_1,x_2)\right)|_{x_1=x_2(1+z)}\\
&=&x_1^{-1}\de\left(\frac{x_2(1+z)}{x_1}\right)\left(\left((x_1-x_2)^kA(x_1,x_2)\right)|_{x_1=\phi_\ep(x_2,x_0)}\right)|_{x_0=f_\ep(x_2,z)}\\
&=&x_1^{-1}\de\left(\frac{x_2(1+z)}{x_1}\right)(x_2z)^kC(f_\ep(x_2,z),x_2),
\end{eqnarray*}
noticing that
$$\phi_{\ep}(x,f_{\ep}(x,z))=x(1+z).$$
Then (\ref{L}) follows.
\end{proof}

As the main result of this section we have:

\begin{theo}\label{Ja}
Let $V$ be a vertex algebra and let $(W, Y_W)$ be a $\phi_\ep$-coordinated module.  Then
\begin{eqnarray}
&&(x_2z)^{-1}\de\left(\frac{x_1-x_2}{x_2z}\right)Y_W(u,x_1)Y_W(v,x_2) -(x_2z)^{-1}\de\left(\frac{x_2-x_1}{-x_2z}\right)Y_W(v,x_2)Y_W(u,x_1)\nonumber\\
&&=x_1^{-1}\de\left(\frac{x_2(1+z)}{x_1}\right)Y_W\left(Y(u, f_\ep(x_2,z))v,x_2\right)
\end{eqnarray}
for $u,v\in V$, where
$$
f_\ep(x,z)=\begin{cases}
&x^{1-\ep}\cdot \frac{(1+z)^{1-\ep}-1}{1-\ep},\quad \ \ \text{for}\ \ep\neq1,\\
&\log (1+z),\quad\quad\quad\quad \text{for}\ \ep=1.
\end{cases}
$$
Furthermore, we have
\begin{eqnarray}\label{Com}
[Y_W(u,x_1),Y_W(v,x_2)]=\sum_{j\geq0 }\frac{1}{j!}\left(x_2^{\ep}\frac{\pa}{\pa x_2}\right)^jx_1^{\ep-1}\de\left(\frac{x_2}{x_1}\right)Y_W(u_{j}v,x_2).
\end{eqnarray}
\end{theo}

\begin{proof}From definition, there exists $k\in\bN$ such that
$$
(x_1-x_2)^kY_W(u,x_1)Y_W(v,x_2)\in\Hom(W,W((x_1, x_2)))
$$
and
$$
(\phi_\ep(x_2,x_0)-x_2)^kY_W(Y(u,x_0)v,x_2)
=\left((x_1-x_2)^kY_W(u,x_1)Y_W(v,x_2)\right)|_{x_1=\phi_\ep(x_2,x_0)}.
$$
On the other hand, by Proposition \ref{P2} we also have
\begin{eqnarray}
(x_1-x_2)^k Y_W(u,x_1)Y_W(v,x_2)=(x_1-x_2)^kY_W(v,x_2)Y_W(u,x_1).
\end{eqnarray}
Then the first assertion follows immediately from Lemma \ref{Le}. Furthermore, applying
$\Res_z x_2$ we get
\begin{eqnarray*}
&&[Y_W(u,x_1),Y_W(v,x_2)]\\
&=&\Res_z x_1^{-1}\de\left(\frac{x_2(1+z)}{x_1}\right)x_2 Y_W\left(Y(u, f_\ep(x_2,z))v,x_2\right)\\
&=&\Res_{x_0} x_1^{-1}\de\left(\frac{\phi_\ep(x_2,x_0)}{x_1}\right)x_2\frac{\partial}{\partial x_0}\left(\frac{\phi_\ep(x_2,x_0)}{x_2}-1\right) Y_W(Y(u,x_0)v,x_2)\\
&=&\Res_{x_0} x_1^{-1}\de\left(\frac{\phi_\ep(x_2,x_0)}{x_1}\right)\phi_{\ep}(x_{2},x_{0})^{\ep}
 Y_W(Y(u,x_0)v,x_2)\\
&=&\Res_{x_0} x_1^{\ep-1}\de\left(\frac{\phi_\ep(x_2,x_0)}{x_1}\right)
 Y_W(Y(u,x_0)v,x_2)\\
&=&\sum_{j\geq0 }\frac{1}{j!}\left[\left(x_2^{\ep}\frac{\pa}{\pa x_2}\right)^jx_1^{\ep-1}\de\left(\frac{x_2}{x_1}\right)\right]Y_W(u_{j}v,x_2),
\end{eqnarray*}
noticing that
$$
\frac{\partial}{\partial x_0}\phi_\ep(x_2,x_0)
=e^{x_{0}x_{2}^{\ep}\frac{\partial}{\partial x_2}}\left(x_{2}^{\ep}\frac{\partial}{\partial x_2}\right)(x_{2})
=e^{x_{0}x_{2}^{\ep}\frac{\partial}{\partial x_2}}(x_{2}^{\ep})
=\left(e^{x_{0}x_{2}^{\ep}\frac{\partial}{\partial x_2}}x_{2}\right)^{\ep}
=\phi_{\ep}(x_{2},x_{0})^{\ep}.
$$
This proves the second assertion.
\end{proof}

\begin{rem}\label{basic-facts}
 {\em We here collect some basic facts that we shall use. We have
\begin{eqnarray}
\left(x_{2}^{\ep}\frac{\partial}{\partial x_{2}}\right)
x_{1}^{\ep-1}\delta\left(\frac{x_{2}}{x_{1}}\right)=-\left(x_{1}^{\ep}\frac{\partial}{\partial x_{1}}\right)
x_{2}^{\ep-1}\delta\left(\frac{x_{1}}{x_{2}}\right),
\end{eqnarray}
\begin{eqnarray}
(x_{1}-x_{2})^{m}\left(x_{2}^{\ep}\frac{\partial}{\partial x_{2}}\right)^{n}
x_{1}^{-1}\delta\left(\frac{x_{2}}{x_{1}}\right)=0
\end{eqnarray}
for any nonnegative integers $m$ and $n$ with $m>n$, and
\begin{eqnarray}
(x_{1}-x_{2})^{n}\left(x_{2}^{\ep}\frac{\partial}{\partial x_{2}}\right)^{n}
x_{1}^{-1}\delta\left(\frac{x_{2}}{x_{1}}\right)
=x_{2}^{n\ep}(x_{1}-x_{2})^{n}\left(\frac{\partial}{\partial x_{2}}\right)^{n}
x_{1}^{-1}\delta\left(\frac{x_{2}}{x_{1}}\right)
\end{eqnarray}
for any nonnegative integer $n$. Furthermore, we have
\begin{eqnarray}
\Res_{x_{1}}x_{1}^{-\ep}(x_{1}-x_{2})^{n}\left(x_{2}^{\ep}\frac{\partial}{\partial x_{2}}\right)^{n}
x_{1}^{\ep-1}\delta\left(\frac{x_{2}}{x_{1}}\right)=\frac{x_{2}^{n\ep}}{n!}.
\end{eqnarray}
These facts can be proved by using the special case with $\ep=0$ (cf. \cite{L1}) and the facts that for any positive integer $n$, there exists polynomials $f_{1}(x),\dots, f_{n}(x)$ such that
\begin{eqnarray}
\left(x_{2}^{\ep}\frac{\partial}{\partial x_{2}}\right)^{n}
=x_{2}^{n\ep}\left(\frac{\partial}{\partial x_{2}}\right)^{n}+f_{1}(x_{2})\left(\frac{\partial}{\partial x_{2}}\right)^{n-1}+\cdots +f_{n}(x_{2}).
\end{eqnarray}}
\end{rem}

The following, which is a generalization of a result in  \cite{L1},
follows immediately from Theorem \ref{Ja} and the basic facts in Remark \ref{basic-facts}:

\begin{lem}\label{faith}
Let $V$ be a vertex algebra and let $(W,Y_{W})$ be a faithful $\phi_{\ep}$-coordinated $V$-module.
Suppose that
\begin{eqnarray}
[Y_W(u,x_1),Y_W(v,x_2)]=\sum_{j\geq0 }\frac{1}{j!}\left(x_2^{\ep}\frac{\pa}{\pa x_2}\right)^jx_1^{\ep-1}\de\left(\frac{x_2}{x_1}\right)Y_W(A^{j},x_2),
\end{eqnarray}
where $u,v, A^{0},A^{1},\dots$ are fixed vectors in $V$. Then $A^{j}=u_{j}v$ for all $j\ge 0$.
\end{lem}

\section{Vertex algebras arising from Novikov algebras }

In this section, we study $\phi_\ep$-coordinated modules for the vertex algebras associated to
Novikov algebras by Primc.

We first recall the definition of a Novikov algebra (see \cite{BN},
\cite{GD1}, \cite{O1}).

\begin{defi}\label{novikov}
{\em A {\em (left) Novikov algebra} is a non-associative algebra $\A$ satisfying
\begin{eqnarray}
(ab)c-a(bc)&=&(ba)c-b(ac),\label{N1}\\
(ab)c&=&(ac)b\label{N2}
\end{eqnarray}
for $a,b,c\in \A$.}
\end{defi}

Note that any commutative and associative algebra is a Novikov
algebra.

%\begin{rem}{\rm
%In the literature, a non-associative algebra $\A$ satisfying (\ref{N1}) is called a
%{pre-Lie algebra} or a {left-symmetric algebra} (cf. \cite{Bu}). }
%\end{rem}

\begin{rem}\label{RG}
{\em  We here recall the Gelfand construction of Novikov algebras  due to S. Gelfand
(see  \cite{GD1}).
Let $A$ be a commutative associative algebra with
a derivation $\pa$. Define a new operation $\circ $ on $A$ by
$a\circ b= a\pa b$ for $a,b\in A$. Then $(A,\circ )$ is a (left) Novikov
algebra.
%{\bf Yufeng: See  page 259 in \cite{GD1}.}
}
\end{rem}

The following result was due to Balinsky and Novikov (see
\cite{BN}):

\begin{pro}\label{BN}
Let $\A$ be a non-associative algebra. Set
\begin{equation}
L(\A)=\A\otimes \bC[t,t^{-1}],\quad \pa=\frac{d}{dt}.
\end{equation}
Define a bilinear operation $[\cdot,\cdot]$  on  $L(\A)$ by
\begin{equation}\label{def1}
[a\ot f, b\ot g]= ab \ot (\pa f)g- ba \ot (\pa g)f
\end{equation}
for $a,b\in \A,\ f,g\in \bC[t,t^{-1}]$. Then $(L(\A),[\cdot,\cdot])$
is a Lie algebra if and only if $\A$ is a Novikov algebra.
\end{pro}

The following refinement was due to Primc (see \cite{Pr}, Example 3; cf. \cite{BN}):

\begin{pro}\label{LC1}
Let $\A$ be a non-associative algebra equipped with a bilinear form
$\lan\cdot,\cdot\ran$. Set
$$
\widetilde L(\A)=L(\A)\oplus \bC \c,
$$
where $\c$ is a distinguished nonzero element.
For $a\in\A,\ m\in\bZ$, set $L(a,m)=a\ot t^{m+1}$. Define a bilinear
operation $[\cdot,\cdot]$ on $\widetilde L(A)$ by
\begin{eqnarray}\label{EC}
[L(a,m),L(b,n)]&=&(m+1)L(ab, m+n)-(n+1)L(ba, m+n)\nonumber\\
&&\ \ \ \ \quad +\frac{1}{12}(m^3-m)\lan a,b\ran \de_{m+n,0}\c,\\
&&[\c, \widetilde L(\A)]=0=[\widetilde L(\A),\c]
\end{eqnarray}
for $a,b\in\A$, $m,n\in\bZ$. Then $(\widetilde L(A),[\cdot,\cdot])$
is a Lie algebra if and only if $\A$ is a Novikov algebra and
$\lan\cdot, \cdot\ran$ is a symmetric form satisfying
\begin{eqnarray}\label{In}
&&\lan ab,c\ran=\lan a,bc\ran,\ \ \ \ \lan ab,c\ran=\lan ba,c\ran
\quad\text{for}\ a,b,c\in\A.
\end{eqnarray}
\end{pro}

\begin{rem}
{\rm Note that a unital Novikov algebra $\A$ with a
symmetric bilinear form $\lan\cdot,\cdot\ran$  satisfying (\ref{In})
amounts to a Frobenius algebra, i.e., a unital commutative and
associative algebra with a nondegenerate symmetric and associative form. For a
Frobenius algebra $\A$, the corresponding Lie algebra $\widetilde L(\A)$ is
isomorphic to the map Virasoro algebra (see \cite{La,Sa}).
In particular, if $\A=\bC e$ is 1-dimensional with $e\cdot e= e$ and
$\lan e,e\ran=\frac{1}{12}$, then $\widetilde L(\A)$ is isomorphic to the Virasoro algebra.
More examples can be found in \cite{PB1,PB2,PB3}. }
\end{rem}

\begin{exa}{\rm
Let $\A=\bC[x,x^{-1}]$ and let $p(x)\in \bC[x,x^{-1}]$. By the
Gelfand construction, one has a Novikov algebra $(\A,\circ_{p(x)})$,
where
$$x^i\circ_{p(x)} x^j
=x^i\left(p(x)\frac{d}{dx}\right)x^j= j x^{i+j-1}p(x)\quad
\text{for}\ i,j\in\bZ.$$ In this case, the corresponding Lie algebra $L(\A)$ is a
Lie algebra of Block type (cf. \cite{B, DZ}). In particular, if
$p(x)=1$, $L(\A)$ is isomorphic to the Poisson
algebra $\bC[x,x^{-1},y,y^{-1}]$ with bracket relation
\begin{eqnarray*}
[f,g]=\frac{\pa f}{\pa y}\frac{\pa g}{\pa x}-\frac{\pa f}{\pa x}\frac{\pa g}{\pa y}
\quad\text{for}\ f,g\in\bC[x,x^{-1},y,y^{-1}].
\end{eqnarray*}
For this special case, taking a basis $L_m^i=x^{i+1+m}y^{m+1}$ for $m,i\in\bZ$, we have
\begin{eqnarray}\label{Poisson}
[L_m^i,L_n^j]=(j(m+1)-i(n+1))L_{m+n}^{i+j}\quad\text{for}\ i,j,m,n\in\bZ.
\end{eqnarray}
The structure and representation theory of this Lie algebra and its
subalgebras have been extensively studied in \cite{Ba, Su, Wi,WT}.}
\end{exa}

Let $\A$ be a Novikov algebra equipped with a symmetric
bilinear form $\lan\cdot,\cdot\ran$  satisfying (\ref{In}).
Following Primc \cite{Pr} we associate vertex algebras to $\A$. For
$a\in \A$, set
\begin{equation}
L(a,x)=\sum_{n\in\bZ} L(a,n)x^{-n-2}\in \widetilde
L(\A)[[x,x^{-1}]].
\end{equation}
In terms of generating functions the relation (\ref{EC}) can be
rewritten as
\begin{eqnarray}\label{GR}
&&[L(a,x_1),L(b,x_2)]\nonumber\\
&=&\left(\frac{\pa}{\pa x_2}L(ba,x_2)\right)x_1^{-1}\de\left(\frac{x_2}{x_1}\right)
+(L(ab,x_2)+L(ba, x_2))\left(\frac{\pa}{\pa x_2}\right)x_{1}^{-1}\de\left(\frac{x_2}{x_1}\right)\nonumber\\
&&\ \ \ \ \quad+\frac{1}{12}\lan a,b\ran \c \left(\frac{\pa}{\pa
x_2}\right)^3x_{1}^{-1}\de\left(\frac{x_2}{x_1}\right)
\end{eqnarray}
for $a,b\in \A$. Set
$$
\widetilde L(\A)_{+}=\A\ot\bC[t]\oplus\bC\c,\quad \widetilde
L(\A)_{-}=\A\ot t^{-1}\bC[t^{-1}].
$$
Note that $\widetilde L(\A)_{\pm}$ are Lie subalgebras and
$\widetilde L(\A)=\widetilde L(\A)_{+}\oplus \widetilde L(\A)_{-}$
as a vector space. Let $\ell\in\bC$ and denote by  $\bC_{\ell}$ the
one-dimensional $ \widetilde L(\A)_{+}$-module with $\c$ acting as
scalar $\ell$ and with $\A\ot \bC[t]$ acting trivially. Form an
induced module
\begin{eqnarray}
V_{\widetilde L(\A)}(\ell,0)=U(\widetilde L(\A))\ot_{U(\widetilde
L(\A)_{+})}\bC_\ell.
\end{eqnarray}
Set ${\bf1}=1\ot 1\in V_{\widetilde L(\A)}(\ell,0)$ and identify $\A$ as a subspace of
$V_{\widetilde L(\A)}(\ell,0)$ through the linear map
$$a \mapsto L(a,-2){\bf 1}\quad\text{for}\ a\in\A.$$
From \cite{Pr} (cf. \cite{DLM,Xu}), there exists a vertex algebra
structure on $V_{\widetilde L(\A)}(\ell,0)$, which is uniquely
determined by the condition that ${\bf 1}$ is the vacuum vector and
$$Y(a, x) = L(a,x)=\sum_{n\in\bZ}L(a,n)x^{-n-2} \quad\text{for}\ a\in \A.$$
Furthermore, $\A$ is a generating subspace of vertex algebra
$V_{\widetilde L(\A)}(\ell,0)$ with
\begin{eqnarray}\label{eva-information}
&&a_0b= \D(ba),\quad a_1b=ab+ba,\quad
a_3b=\frac{1}{2}\ell \lan a,b \ran{\bf 1},\quad a_2b=0=a_kb\ \ \ \
\label{V3}
\end{eqnarray}
for $a,b\in \A$ and for $k\geq 4$.

Next, we discuss a $\bZ$-graded vertex algebra structure on $V_{\widetilde L(\A)}(\ell,0)$.

\begin{defi}\label{def-zgva}
{\em A {\em $\bZ$-graded vertex algebra} is a vertex algebra $V$ equipped with a $\bZ$-grading
$V=\oplus_{n\in \bZ}V_{n}$ such that ${\bf 1}\in V_{0}$ and
\begin{eqnarray}
u_{k}v\in V_{m+n-k-1}\ \ \ \mbox{ for }u\in V_{m},\ v\in V_{n},\ m,n,k\in \bZ.
\end{eqnarray}}
\end{defi}
%As a consequence of the definition, we have ${\bf 1}\in V_{0}$.

Let $(\A,\lan\cdot,\cdot\ran)$ be given as before.
It can be readily seen that $\widetilde{L}(\A)$ is a $\bZ$-graded Lie algebra with
$\deg \c=0$ and
\begin{eqnarray}
\deg (a\otimes t^{m})=\deg (L(a,m-1))=-m+1\ \ \ \mbox{ for }a\in \A,\ m\in \bZ.
\end{eqnarray}
As $\widetilde L(\A)_{+}$ is a graded subalgebra, $V_{\widetilde L(\A)}(\ell,0)$ is naturally a $\bZ$-graded $\widetilde L(\A)$-module with $\deg {\bf 1}=0$ and with $V_{\widetilde L(\A)}(\ell,0)_{2}=\A$.
Furthermore, by Lemma A (in Appendix) $V_{\widetilde L(\A)}(\ell,0)$ equipped with this $\bZ$-grading is a $\bZ$-graded vertex algebra. In view of the PBW Theorem,
 $\D a=a_{-2}{\bf 1}\ne 0$ for any nonzero $a\in \A$ and $V_{\widetilde L(\A)}(\ell,0)$
is linearly spanned by the vectors
$$a^{(1)}_{-m_1}\cdots a^{(r)}_{-m_r}{\bf 1}$$
for $r\geq 0, \ a^{(i)}\in \A,\ m_i\geq 1$.

On the other hand, we have (also see \cite{BKL}):

\begin{pro}\label{pcharacterization}
Let $V=\oplus_{n\in \bZ}V_{n}$ be a $\bZ$-graded vertex algebra with
the following properties:
\begin{itemize}
\item[(1)] $V_n=0\quad \text{for}\  n< 0$, $\ V_0=\bC{\bf 1},\quad V_1=0$;
\item[(2)] $(\Ker \D)\cap V_{2}=0$, where $\D$ is the linear operator on $V$ defined by $\D v= v_{-2}{\bf 1}$;
\item[(3)]
$ V={\rm span}\{a^{(1)}_{-m_1}\cdots a^{(r)}_{-m_r}{\bf 1}\ | \
r\geq 0, \ a^{(i)}\in V_2,\ m_i\geq 1\}. $
\end{itemize}
Then there exist a bilinear operation $*$ on $V_2$, uniquely determined by
$$b_{0}a=\D(a*b)\quad\text{for}\ a,b\in V_2,$$
and a bilinear form $\lan\cdot, \cdot\ran$ on $V_2$, uniquely determined by
$$\quad a_3b=\frac{1}{2}\lan a,b\ran{\bf 1} \quad\text{for}\ a,b\in V_2.$$
Furthermore, $(V_2,*)$ is a Novikov algebra  and  $\lan\cdot,\cdot\ran$ is a
symmetric form satisfying (\ref{In}).
\end{pro}

\begin{proof} As $V=\oplus_{n\in \bZ}V_{n}$ is a $\bZ$-graded vertex algebra, we have
$$
a_0b\in V_3,\quad a_1b\in V_2,\quad a_2b\in V_1=0, \quad a_3b \in
V_0=\bC{\bf 1},\quad a_nb=0
$$
for $a,b\in V_2$ and $n\geq 4$. From the span property, we have
$$
V_3=\{a_{-2}{\bf 1}\,|\,a\in V_2\}=\D V_2.
$$
 Since $(\Ker \D)\cap V_{2}=0$, it follows that $*$
is a well-defined operation on $V_2$. Moreover, using the
skew-symmetry of $V$, we get
$$b_0a= -a_0b+\D(b_1a),\quad a_3b=b_3a\quad \text{for }\ a,b\in V_2.
$$
It then follows that
\begin{eqnarray}
a*b+b*a=b_{1}a=a_{1}b,\quad \lan a,b\ran=\lan b,a\ran\quad \text{for
}\ a,b\in V_2.
\end{eqnarray}
Set $\widetilde{V}=V\otimes \bC[t,t^{-1}]$ and $\tilde{\D}=\D\otimes
1+1\otimes \frac{d}{dt}$. It was known that $\widetilde{V}/\tilde{\D}\widetilde{V}$ is a Lie algebra
where for $u,v\in V,\ m,n\in \bZ$,
$$[\overline{u\otimes t^{m}},\overline{v\otimes t^{n}}]
=\sum_{j\ge 0}\binom{m}{j}\overline{u_{j}v\otimes t^{m+n-j}}.$$
For $a,b\in V_2,\ m,n\in \bZ$, we have
\begin{eqnarray*}
&&[a_m,b_n]\\
&=&\sum_{i\geq0}\binom{m}{i} (a_ib)_{m+n-i}\\
&=&(a_0b)_{m+n}+m(a_1b)_{m+n-1}+\frac{1}{6}m(m-1)(m-2)(a_3b)_{m+n-3}\\
&=&(\D(b*a))_{m+n}+m(a*b+b*a)_{m+n-1}+\frac{1}{6}m(m-1)(m-2)(a_3b)_{m+n-3}\\
&=&-(m+n)(b*a)_{m+n-1}+m(a*b+b*a)_{m+n-1}+\frac{1}{6}m(m-1)(m-2)(a_3b)_{m+n-3}\\
&=&m(a*b)_{m+n-1}-n(b*a)_{m+n-1}+\frac{1}{12}m(m-1)(m-2)\lan a,b\ran {\bf 1}_{m+n-3}\\
&=&m(a*b)_{m+n-1}-n(b*a)_{m+n-1}+\frac{1}{12}m(m-1)(m-2)\lan a,b\ran
\de_{m+n,2}.
\end{eqnarray*}
From this and the assumption (2), we see that the non-associative algebra
$\widetilde{L}(V_{2},*)$ defined in Proposition \ref{LC1} is a
subalgebra of $\widetilde{V}/\tilde{\D}\widetilde{V}$. In view of
Proposition \ref{LC1}, $(V_2,*)$ is a Novikov algebra and
$\lan\cdot,\cdot\ran$ is a symmetric bilinear form satisfying
(\ref{In}).
\end{proof}

\begin{rem}
{\em Let $V$ be a vertex algebra. Suppose $v\in \Ker \D\subset V$.
Then
$$\frac{d}{dx}Y(v,x)=Y(\D v,x)=0,$$
which implies that $v_{n}=0$ for all $n\ne -1$. Consequently, $v$
lies in the center of $V$ and $[\D, v_{-1}]=v_{-2}=0$. It then follows that $v_{-1}V$ is an ideal of
$V$. If $V$ is $\bN$-graded and if $v\in (\Ker \D)\cap V_{n}$ with
$n\ge 1$, then $v_{-1}V$ is a proper ideal. Thus, if $V$ is a
graded simple $\bN$-graded vertex algebra, we have $(\Ker
\D)\cap V_{n}=0$ for all $n\ge 1$.}
\end{rem}

\begin{rem}
{\em  The bilinear operation $*$ on $V_2$ was used
by Dijkgraaf in \cite{Di} in his study on the genus one partition function,
which is controlled by a contact term pre-Lie algebra
given in terms of the operator product expansion.}
\end{rem}

We next discuss quasi vertex operator algebras, or namely M\"obius vertex algebras.
Fix a basis $\{L(1),L(0),L(-1)\}$ for
$sl(2,\bC)$ such that
$$[L(0),L(\pm 1)]=\mp L(\pm 1),\ \ \ \ [L(1),L(-1)]=2L(0).$$
 A {\em M\"obius vertex algebra} (see \cite{FHL}) is a
$\bZ$-graded vertex algebra $V=\oplus_{n\in \bZ}V_{n}$, equipped
with a representation of $sl(2,\bC)$ on $V$ such
that $V_{n}=0$ for $n$ sufficiently negative,
$$ L(0)|_{V_{n}}=n\ \
\mbox{ for }n\in \bZ,$$ and
\begin{eqnarray}
&&[L(-1),Y(v,x)]=Y(L(-1)v,x)=\frac{d}{dx}Y(v,x),\\
&&[L(1),Y(v,x)]=Y(L(1)v,x)+2xY(L(0)v,x)+x^2Y(L(-1)v,x)
\end{eqnarray}
for $v\in V$. Note that for a M\"obius vertex algebra $V$, we have
$L(-1)=\D$ on $V$.

We have (cf. \cite{HL}):

\begin{pro}\label{pcharacterization-qvoa}
Let $V=\oplus_{n\in \bZ}V_{n}$ be a M\"obius vertex algebra
satisfying all the conditions given in Proposition
\ref{pcharacterization}. Then $(V_{2},*)$ is a commutative and
associative algebra.
\end{pro}

\begin{proof}
For $v\in V_{2}$, as $L(1)v\in V_{1}=0$, we have
\begin{eqnarray}\label{eL1L-1}
L(1)L(-1)v=L(-1)L(1)v+2L(0)v=4v.
\end{eqnarray}
This implies
$$(\ker L(-1)) \cap V_2=0.$$
On the other hand, for $v\in V_{2}$, from \cite{FHL} we have
\begin{eqnarray}\label{eL(1)vm-comm}
[L(1),v_m]=(-m+2)v_{m+1}\quad\text{for}\ m\in\bZ.
\end{eqnarray}
Let $u,v\in V_{2}$. Using  (\ref{eL1L-1}), the definition of
$u*v$, and (\ref{eL(1)vm-comm}), we get
\begin{eqnarray*}
4u*v=L(1)L(-1)(u*v) =L(1)(v_{0}u)=v_{0}L(1)u-2v_{1}u=2v_{1}u,
\end{eqnarray*}
which gives $u*v=\frac{1}{2}v_{1}u$.  On the other hand, using skew symmetry we get
$$u_{1}v=v_{1}u-L(-1)v_{2}u+\frac{1}{2}L(-1)^{2}v_{3}u+\cdots=v_{1}u.$$
Therefore $u*v=v*u$. This proves
that $(V_{2},*)$ is commutative. Consequently, $(V_{2},*)$ is
commutative and associative.
\end{proof}

Furthermore, we have:

\begin{pro}\label{MC}
Let $\A$ be a Novikov algebra equipped with
a symmetric bilinear form $\lan\cdot,\cdot\ran$  satisfying
(\ref{In}). Then for any $\ell\in \bC$, the $\bZ$-graded vertex algebra $V_{\widetilde
L(\A)}(\ell,0)$ has a compatible M\"obius vertex algebra
structure if and only if $\A$ is commutative and associative.
\end{pro}

\begin{proof} The ``only if'' part follows from Proposition \ref{pcharacterization-qvoa}.
For the ``if'' part we assume that $\A$ is a commutative and associative algebra with
a symmetric bilinear form $\lan\cdot,\cdot\ran$  satisfying
(\ref{In}).
 First, by Lemma C, $sl_{2}$ acts on $\widetilde{L}(\A)$ by derivations, where $sl_{2}\cdot \c=0$ and
 \begin{eqnarray*}
L(-1+j)(a\otimes t^{n})=(j-n)(a\otimes t^{n+j-1})
\end{eqnarray*}
for $j=0,1,2$ and for $a\in \A,\ n\in \bZ$.
Then $sl_{2}$ acts on the universal enveloping algebra $U(\widetilde{L}(\A))$ as a Lie algebra of derivations. We see that the action of $sl_{2}$ preserves the subalgebra $\widetilde{L}(\A)_{+}$.
It follows from the construction of $V_{\widetilde
L(\A)}(\ell,0)$ that $sl_{2}$ acts on $V_{\widetilde
L(\A)}(\ell,0)$ with $sl_{2}\cdot {\bf 1}=0$. For $a\in \A$, we have
\begin{eqnarray*}
&&[L(-1),Y(a,x)]=\frac{d}{dx}Y(a,x),\\
&&[L(0),Y(a,x)]=x\frac{d}{dx}Y(a,x)+2Y(a,x),\\
&&[L(1),Y(a,x)]=4xY(a,x)+x^{2}\frac{d}{dx}Y(a,x).
\end{eqnarray*}
Then by Lemma B (in Appendix) $V_{\widetilde
L(\A)}(\ell,0)$ is a M\"obius vertex algebra.
\end{proof}

\begin{rem}
{\rm Let $\A=\bC e_1\oplus \bC e_2$ with a multiplicative operation $\circ$ given by
$$
e_1\circ e_1= e_1+e_2,\quad e_2\circ e_1=e_2,\quad e_1\circ
e_2=e_2\circ e_2=0.
$$
This is a noncommutative and nonassociative Novikov algebra.
Furthermore, the bilinear form $\lan\cdot,\cdot\ran$, defined by
$$\lan e_1,e_1\ran=\frac{1}{12},\quad \lan e_1,e_2\ran=\lan
e_2,e_1\ran=\lan e_2,e_2\ran=0,$$
is (degenerate) symmetric and satisfies (\ref{In}).
In view of Proposition \ref{MC}, $V_{\widetilde L(\A)}(\ell,0)$
is not a M\"obius vertex algebra.
The corresponding Lie algebra $L(\A)$ has been
extensively studied in \cite{PB3}.}
\end{rem}

Next we study $\phi_\ep$-coordinated modules for vertex algebra $V_{\widetilde{L}(\A)}(\ell,0)$.
First, we construct certain infinite-dimensional Lie algebras, generalizing Lie algebra $\widetilde{L}(\A)$.

\begin{lem}\label{LN}
Let $\A$ be a Novikov algebra and let $\K$ be a commutative and associative algebra with a derivation $\partial$. Define a bilinear operation $[\cdot,\cdot]$ on $\A\ot\K$ by
\begin{equation}\label{def2}
[a\ot f, b\ot g]= ab \ot (\pa f)g- ba \ot (\pa g)f
\end{equation}
for $a,b\in \A,\  f,g\in \K$. Then
$(\A\ot\K,[\cdot,\cdot])$ is a Lie algebra.
\end{lem}

\begin{proof}
It is straightforward. Alternatively, it follows from a general result in algebraic operad theory (see \cite{LV} and \cite{Za}) as follows:
First, define a new operation $*$ on $\K$ by
 $$f*g=(\pa f) g\quad\text{for}\ f,g\in \K.$$
Then $(\K,*)$ is a right Novikov algebra. Note that from \cite{Dr} (Theorem 1.3),
left Novikov algebras and right Novikov algebras are algebras over binary quadratic operads dual to each other. It follows from \cite{GK} (Theorem 2.2.6 (b)) that $(\A\ot\K,[\cdot,\cdot])$ is a Lie algebra.
\end{proof}

In view of Lemma \ref{LN}, for any Novikov algebra $\A$ and for any integer $\ep$ we have a Lie algebra
$L^\ep(\A)$, where
\begin{equation}
L^{\ep}(\A)=\A\otimes \bC[t,t^{-1}]
\end{equation}
as a vector space and the bilinear operation $[\cdot,\cdot]$ is given by
\begin{equation}\label{def}
[a\ot f, b\ot g]= ab \ot \left(t^{\ep}\frac{d}{dt}f\right)g- ba \ot \left(t^{\ep}\frac{d}{dt}g\right)f
\end{equation}
for $a,b\in \A,\  f,g\in \bC[t,t^{-1}]$.

\begin{rem}
{\em Assume that $\A$ is a commutative Novikov algebra. It can be readily seen that
the linear map $\theta: L(\A)\to L^{\ep}(\A)$ defined by $\theta(L(a,m))=L^\ep(a,m)$ for $a\in\A, m\in\bZ$ is an isomorphism of Lie algebras. Furthermore, if $\A$ is unital, that is, $\A$ is a unital commutative and associative algebra, one can show that $\widetilde{L}(\A)\simeq \widetilde{L}^{\ep}(\A)$.}
\end{rem}

\begin{rem}
{\rm Let $\A=\bC[z,z^{-1}]$ with a derivation $z\frac{d}{dz}$. Define
$$z^i\circ z^j=z^i(z\frac{d}{dz})(z^j)= j z^{i+j}\quad \text{for}\ i,j\in\bZ.$$
Then we have a Novikov algebra $(\A,\circ)$. Furthermore, for any $\ep\in \bZ$
we have a Lie algebra $L^{\ep}(\A)$.
Denote $L^{\ep}(i,m)=L^{\ep}(z^i,m)\in L^{\ep}(\A)$ for $i,m\in\bZ$.
Then
$$[L^\ep(i,m),L^\ep(j,n)]=(j(m+1-\ep)-i(n+1-\ep))L^\ep(i+j,m+n)\quad\text{for}\ i,j,m,n\in\bZ.
$$
Note that $L^{0}(\A)$ is isomorphic to the poisson  Poisson Lie algebra defined as in (\ref{Poisson}).
On the other hand, $L^{1}(\A)$ is isomorphic to the Lie algebra of area-preserving diffeomorphisms of
the two-torus investigated by V. Arnold in \cite{Ar},
which is generated by $L_{m}^{i}$ with $m,i\in \bZ$, subject to relations
\begin{eqnarray}
[L_m^i,L_n^j]=(jm-in)L_{m+n}^{i+j}\quad\text{for}\ i,j,m,n\in\bZ.
\end{eqnarray}
It was also called the Virasoro-like algebra in \cite{KPS}.
Note that from \cite{DZ}, Lie algebra $L^{1}(\A)$ is {\em not} isomorphic to $L(\A)$.}
\end{rem}

\begin{pro} Let $\A$ be a Novikov algebra with a symmetric bilinear
form $\lan\cdot, \cdot\ran$ satisfying
(\ref{In}). Set
\begin{eqnarray}
\widetilde L^\ep(\A)=L^\ep(\A)\oplus \bC \c_\ep,
\end{eqnarray}
where $\c_\ep$ is a nonzero element.
For $a\in\A,\ m\in\bZ$, denote $L^\ep(a,m)=a\ot t^{m+1-\ep}$.
Then $\widetilde L^\ep(\A)$ is a Lie algebra with
\begin{eqnarray}\label{EC1}
&&[L^\ep(a,m),L^\ep(b,n)]\nonumber\\
&=&(m+1-\ep)L^\ep(ab, m+n)-(n+1-\ep)L^\ep(ba, m+n)\nonumber\\
&&+\frac{1}{12}(m+1-\ep)m(m-1+\ep)\lan a,b\ran \de_{m+n,0}\c_\ep
\end{eqnarray}
for $a,b\in\A$, $m,n\in\bZ$, and with $\c_\ep$ central.
\end{pro}

\begin{proof}
Define a bilinear form $(\cdot,\cdot)$ on the Lie algebra $L^{\ep}(\A)$ by
$$\left(L^{\ep}(a,m),L^{\ep}(b,n)\right)
=\frac{1}{12}(m+1-\ep)m(m-1+\ep)\lan a,b\ran \de_{m+n,0}$$
for $a,b\in A,\ m,n\in \bZ$.
Notice that
$$(m+1-\ep)m(m-1+\ep)=m(m^{2}-(1-\ep)^{2}),$$
which is an odd function of $m$.
As $\lan\cdot,\cdot\ran$ is symmetric, $(\cdot,\cdot)$ is skew symmetric.
% For $a,b\in\A$, $m,n\in\bZ$ with $m+n=0$,
%\begin{eqnarray*}
%&&(L^\ep(a,m),L^\ep(b,n))+(L^\ep(b,n),L^\ep(a,m))\\
%&=&\frac{1}{12}(m+1-\ep)m(m-1+\ep)\lan a,b\ran +\frac{1}{12}(n+1-\ep)n(n-1+\ep)\lan b,a,\ran\\
%&=&\frac{1}{12}(m+1-\ep)m(m-1+\ep)\lan a,b\ran -\frac{1}{12}(m-1+\ep)m(m+1-\ep)\lan a,b\ran\\
%&=&0.
%\end{eqnarray*}

For cocycle condition, let $a,b,c\in\A$, $m,n,k\in\bZ$. We have
\begin{eqnarray*}
&&([L^\ep(a,m),L^\ep(b,n)],L^\ep(c,k))\\
&=&\frac{1}{12}(m+1-\ep)(m+n)\left((m+n)^{2}-(1-\ep)^{2}\right)\lan ab,c\ran \de_{m+n+k,0}\\
&&-\frac{1}{12}(n+1-\ep)(m+n)\left((m+n)^{2}-(1-\ep)^{2}\right)\lan ba,c\ran \de_{m+n+k,0}\\
&=&\frac{1}{12}(m-n)(m+n)\left((m+n)^{2}-(1-\ep)^{2}\right)\lan ab,c\ran \de_{m+n+k,0}\\
&=&\frac{1}{12}(m^{2}-n^{2})\left(k^{2}-(1-\ep)^{2}\right)\lan ab,c\ran \de_{m+n+k,0},
\end{eqnarray*}
where we used the property $\lan ab,c\ran=\lan ba,c\ran$.
Furthermore, we have
$$\lan bc,a\ran=\lan a,bc\ran=\lan ab,c\ran,\ \ \
\lan ca,b\ran=\lan c,ab\ran=\lan ab,c\ran,$$
and
$$(m^{2}-n^{2})\left(k^{2}-(1-\ep)^{2}\right)+(n^{2}-k^{2})\left(m^{2}-(1-\ep)^{2}\right)
+(k^{2}-m^{2})\left(n^{2}-(1-\ep)^{2}\right)=0.$$
Then the cocycle condition follows immediately.
Therefore, $\widetilde L^\ep(\A)$ is a Lie algebra.
\end{proof}

Note that
$$\widetilde L^0(\A)=\widetilde L(\A).$$

For $a\in \A$, set
\begin{equation}
L^{\ep}(a,x)=\sum_{n\in\bZ} L^{\ep}(a,n)x^{-n-2+2\ep}\in \widetilde{L^\ep}(\A)[[x,x^{-1}]].
\end{equation}
In terms of generating functions the relation (\ref{EC1}) can be written as
\begin{eqnarray}\label{GR2}
&&[L^{\ep}(a,x_1),L^{\ep}(b,x_2)]\nonumber\\
&=&\left(x_2^\ep\frac{\pa}{\pa x_2}L^\ep(b a,x_2)\right)x_1^{-1+\ep}\de\left(\frac{x_2}{x_1}\right)\nonumber\\
&&+(L^\ep(ab,x_2)+L^\ep(ba, x_2))\left(x_2^{\ep}\frac{\pa}{\pa x_2}\right)x_{1}^{-1+\ep}\de\left(\frac{x_2}{x_1}\right)\nonumber\\
&&+\frac{1}{12}\lan a,b\ran\c_\ep \left(x_2^{\ep}\frac{\pa}{\pa x_2}\right)^3x_{1}^{-1+\ep}\de\left(\frac{x_2}{x_1}\right)
\end{eqnarray}
for $a,b\in \A$.

\begin{defi}
{\em An $\widetilde L^\ep(\A)$-module on which $\c_\ep$ acts as a scalar $\ell\in\bC$ is said to be of {\em level} $\ell$.
An $\widetilde L^\ep(\A)$-module $W$ is said to be {\em restricted} if $L^\ep(a,x)w\in W((x))$
for every $a\in \A,\ w\in W$.
Denote by $L_W^\ep(a,x)$ the corresponding element of $\E(W)$.}
\end{defi}

As the main result of this section we have:

\begin{theo}
Let $\ell\in \bC$ and
let $W$ be a restricted $\widetilde L^\ep(\A)$-module of level $\ell$. Then there exists a
$\phi_\ep$-coordinated  $V_{\widetilde L(\A)}(\ell,0)$-module structure on $W$, which is uniquely
determined by $Y_W(a,x)=L_W^\ep(a,x)\ \text{for}\ a\in\A.$  On the other hand,  let $(W, Y_W)$ be a $\phi_\ep$-coordinated $V_{\widetilde L(\A)}(\ell,0)$-module. Then $W$ is a restricted $\widetilde L^\ep(\A)$-module of level $\ell$, which is given by $L_W^\ep(a,x)=Y_W(a,x)$ for $a\in\A.$
\end{theo}

\begin{proof}
Assume that $W$ is a restricted $\widetilde L^\ep(\A)$-module of level
$\ell$. Set
$$
U_W=\text{span}\{L_W^\ep(a,x)\mid a\in \A\}\subset \E(W).
$$
For $a,b\in \A$, from (\ref{GR}) we have
$$(x_1-x_2)^4[L_W^\ep(a,x_1),L_W^\ep(b,x_2)]=0.$$
Thus $U_W$ is local. By  Theorem \ref{GT}, $U_W$ generates a vertex algebra $\lan U_W\ran_{\phi_{\ep}}$ and $W$ is a
faithful $\phi_\ep$-coordinated $\lan U_W\ran_{\phi_{\ep}}$-module with
$$Y_{W}(\alpha(x),z)=\alpha(z)\ \ \mbox{ for }\alpha(x)\in \lan U_W\ran_{\phi_{\ep}}.$$
Using the commutation relation of $\widetilde L^\ep(\A)$ we have
\begin{eqnarray}
&&[Y_{W}\left(L_{W}^{\ep}(a,x),x_{1}\right),Y_{W}\left(L_{W}^{\ep}(b,x),x_{2}\right)]\nonumber\\
&=&[L_{W}^{\ep}(a,x_{1}),L_{W}^{\ep}(b,x_{2})]\nonumber\\
&=&\left(x_2^\ep\frac{\pa}{\pa x_2}L_{W}^\ep(b a,x_2)\right)x_1^{-1+\ep}\de\left(\frac{x_2}{x_1}\right)\nonumber\\
&&+(L_{W}^\ep(ab+ba, x_2))\left(x_2^{\ep}\frac{\pa}{\pa x_2}\right)x_{1}^{-1+\ep}\de\left(\frac{x_2}{x_1}\right)\nonumber\\
&&+\frac{1}{12}\lan a,b\ran\ell \left(x_2^{\ep}\frac{\pa}{\pa x_2}\right)^3x_{1}^{-1+\ep}\de\left(\frac{x_2}{x_1}\right)\nonumber\\
&=&\left(x_2^\ep\frac{\pa}{\pa x_2}Y_{W}\left(L_{W}^\ep(b a,x),x_2\right)\right)x_1^{-1+\ep}\de\left(\frac{x_2}{x_1}\right)\nonumber\\
&&+Y_{W}\left(L_{W}^\ep(ab+ba,x), x_2\right)\left(x_2^{\ep}\frac{\pa}{\pa x_2}\right)x_{1}^{-1+\ep}\de\left(\frac{x_2}{x_1}\right)\nonumber\\
&&+\frac{1}{12}\lan a,b\ran\ell \left(x_2^{\ep}\frac{\pa}{\pa x_2}\right)^3x_{1}^{-1+\ep}\de\left(\frac{x_2}{x_1}\right).
\end{eqnarray}
In view of Lemma \ref{faith}, we have
\begin{eqnarray*}
&&L_{W}^{\ep}(a,x)_{0}^{\ep}L_{W}^{\ep}(b,x)=\D L_{W}^{\ep}(ba,x),\ \ \ \
L_{W}^{\ep}(a,x)_{1}^{\ep}L_{W}^{\ep}(b,x)=L_{W}^{\ep}(ab+ba,x),\\
&&L_{W}^{\ep}(a,x)_{3}^{\ep}L_{W}^{\ep}(b,x)=\frac{1}{2}\ell \lan a,b\ran 1_{W},\ \ \ \
L_{W}^{\ep}(a,x)_{j}^{\ep}L_{W}^{\ep}(b,x)=0
\end{eqnarray*}
for $j=2$ and for $j\ge 4$.
Then by Theorem \ref{Ja} we have
\begin{eqnarray*}
&&[Y_{\E}^\ep(L_W^\ep(a,x),x_1),Y_{\E}^{\ep}(L_W^\ep(b,x),x_2)]\nonumber\\
&=&Y_{\E}^{\ep}\left(\D L_W^\ep(b a,x),x_2\right)x_1^{-1}\de\left(\frac{x_2}{x_1}\right)+Y_{\E}^{\ep}(L_W^\ep(ab+ba,x),x_2)\left(\frac{\pa}{\pa x_2}\right)x_1^{-1}\de\left(\frac{x_2}{x_1}\right)\\
&&+\frac{1}{12}\lan a,b\ran\ell\left(\frac{\pa }{\pa x_2}\right)^3x_1^{-1}\de\left(\frac{x_2}{x_1}\right)
\end{eqnarray*}
for $a,b\in\A$. This shows that $\lan  U_W\ran_{\phi_{\ep}}$ is an $\widetilde L(\A)$-module of level $\ell$ with $L(a,x_1)$ acting as $Y_\E^{\ep}(L^\ep_W(a,x),x_1)$ for $a\in \A$ and
\begin{eqnarray*}
&&L(a,n)1_W=L^\ep_W(a,x)_n^{\ep}1_W=0\quad\text{for}\  a\in \A,\ n\in \bN.
\end{eqnarray*}
From the construction of $V_{\widetilde L(\A)}(\ell,0)$, there exists an $\widetilde L(\A)$-module homomorphism  $\rho$ from
$V_{\widetilde L(\A)}(\ell,0)$ to $\lan  U_W\ran_{\phi_{\ep}}$ with
$\rho({\bf 1})=1_W$. That is,
$$
\rho(Y(a,x_1)v))=Y_{\E}^{\ep}(L_W^\ep(a,x),x_1)\rho(v)
\quad\text{for}\ a\in \A,\ v\in V_{\widetilde L(\A)}(\ell,0).
$$
Since the vertex algebra $\lan  U_W\ran_{\phi_{\ep}}$  is generated by $L_W^\ep(a,x)$ for $a\in \A$,   it follows that $\rho$ is a  homomorphism of vertex algebras. As $W$ is a $\phi_\ep$-coordinated module for
$\lan  U_W\ran_{\phi_{\ep}}$, $W$ is  a $\phi_\ep$-coordinated  $V_{\widetilde L(\A)}(\ell,0)$-module through homomorphism $\rho$. Therefore $W$ is a
$\phi_\ep$-coordinated  $V_{\widetilde L(\A)}(\ell,0)$-module.

On the other hand,  let $(W, Y_W)$ be a $\phi_\ep$-coordinated $V_{\widetilde L(\A)}(\ell,0)$-module.
 Using the relations (\ref{eva-information}) and the commutator formula (\ref{Com}) for $\phi_\ep$-coordinated modules for vertex algebras in Theorem \ref{Ja}, we have
\begin{eqnarray*}
&&[Y_{W}(a,x_1),Y_{W}(b,x_2)]\nonumber\\
&=&Y_{W}(\D (b a),x_2)x_1^{-1+\ep}\de\left(\frac{x_2}{x_1}\right)+Y_{W}(ab+b a,x_2)\left(x^\ep\frac{\pa}{\pa x_2}\right)x_1^{-1+\ep}\de\left(\frac{x_2}{x_1}\right)\\
&&+\frac{\ell}{12}\lan a,b\ran \left(x_2^{\ep}\frac{\pa}{\pa x_2}\right)^3x_{1}^{-1+\ep}\de\left(\frac{x_2}{x_1}\right)\\
&=&\left(x^\ep\frac{\pa}{\pa x}\right)Y_{W}(b a,x_2)x_1^{-1+\ep}\de\left(\frac{x_2}{x_1}\right)+Y_{W}(ab+ba,x_2)\left(x^\ep\frac{\pa}{\pa x_2}\right)x_1^{-1+\ep}\de\left(\frac{x_2}{x_1}\right)\\
&&+\frac{\ell}{12}\lan a,b\ran \left(x_2^{\ep}\frac{\pa}{\pa x_2}\right)^3x_{1}^{-1+\ep}\de\left(\frac{x_2}{x_1}\right)
\end{eqnarray*}
for $a,b\in\A$, where we use the fact (see \cite{L6})
$$Y_{W}(\D v,x)=\left(x^{\ep}\frac{\partial}{\partial x}\right)Y_{W}(v,x)
\ \ \mbox{ for }v\in V_{\widetilde L(\A)}(\ell,0).$$
 This proves that $W$ is a restricted module for
$\widetilde L^\ep(\A)$ of level $\ell$.
\end{proof}

\section*{Appendix}
We here establish some basic results we needed in the main body of the paper.

\begin{lemma}[A]
Let $V$ be a vertex algebra equipped with a $\bZ$-grading $V=\oplus_{n\in \bZ}V_{n}$. Suppose that
$U$ is a graded subspace such that $V$ as a vertex algebra is generated by $U$ and
\begin{eqnarray}
u_{r}V_{n}\subset V_{m+n-r-1}\ \ \ \mbox{ for }u\in U\cap V_{m},\ m,n,r\in \bZ.
\end{eqnarray}
Then $V$ is a $\bZ$-graded vertex algebra.
\end{lemma}

\begin{proof} From the definition we need to prove that for every $v\in V_{m}$ with $m\in \bZ$ and for every $n\in \bZ$, $v_{n}$ is a homogeneous operator of degree $m-n-1$. Let $K$ be the linear span
of homogeneous vectors $v\in V$ such that
$$\deg v_{n}=\deg v-n-1\ \ \ \mbox{ for all }n\in \bZ.$$
Now we must prove $K=V$.  By assumption we have $U\subset K$ and it is clear that ${\bf 1}\in K$.
Recall the iterate formula: For $a,b\in V,\ m,n\in \bZ$,
\begin{eqnarray}\label{eiterate-formula}
(a_{m}b)_{n}=\sum_{i\ge 0}\binom{m}{i}(-1)^{i}\left( a_{m-i}b_{n+i}-(-1)^{m}b_{m+n-i}a_{i}\right).
\end{eqnarray}
It follows from this formula that $K$ is a graded vertex subalgebra. As $U$ generates $V$, we must have $K=V$.
\end{proof}

\begin{lemma}[B]
Let $V=\oplus_{n\in \bZ}V_{n}$ be a $\bZ$-graded vertex algebra and an $sl_{2}$-module
such that $sl_{2}\cdot {\bf 1}=0$. Suppose that
$U$ is a graded subspace such that $V$ as a vertex algebra is generated by $U$ and
\begin{eqnarray*}
&&[L(-1),u_{n}]=-nu_{n-1},\ \ \ \ \ \
[L(0),u_{n}]=(\deg u-n-1)u_{n},\\
&&[L(1),Y(u,x)]=Y\left((L(1)+2xL(0)+x^{2}L(-1))u,x\right)
\end{eqnarray*}
for homogeneous $u\in U$ and for every integer $n$.
Then $V$ is a M\"obius vertex algebra.
\end{lemma}

\begin{proof} Just as in the proof of Lemma A, using (\ref{eiterate-formula}) we get
$\deg v_{n}=\deg v-n-1$
for every homogeneous vector $v\in V$ and for every integer $n$. It follows
that $L(0)|_{V_{m}}=m$ for $m\in \bZ$ as
$$L(0)v=L(0)v_{-1}{\bf 1}=v_{-1}L(0){\bf 1}+mv_{-1}{\bf 1}=mv\ \ \mbox{ for }v\in V_{m}.$$
Note that by assumption, we have $[L(-1),Y(u,x)]=\frac{d}{dx}Y(u,x)$ for $u\in U$.
Assume $$[L(-1),Y(a,x)]=\frac{d}{dx}Y(a,x),\ \ \ \ \
[L(-1),Y(b,x)]=\frac{d}{dx}Y(b,x)$$
for some $a,b\in V$. Using the vertex-operator form of (\ref{eiterate-formula}),
as in \cite{L2} (Lemma 3.1.8) we get
$$[L(-1),Y(Y(a,x_{0})b,x)]=\frac{\partial}{\partial x}Y(Y(a,x_{0})b,x).$$
Then it follows that
 $[L(-1),Y(v,x)]=\frac{d}{dx}Y(v,x)$ for all $v\in V$.

Next, we consider $L(1)$.
Set $A(x)=L(1)+2xL(0)+x^{2}L(-1)$. Suppose
 $$[L(1),Y(a,x)]=Y(A(x)a,x),\ \ \ \ [L(1),Y(b,x)]=Y(A(x)b,x)$$
for some $a,b\in V$. Then
\begin{eqnarray*}
&&[L(1),Y(Y(a,x_{0})b,x_{2})]\\
&=&\Res_{x_{1}}x_{0}^{-1}\delta\left(\frac{x_{1}-x_{2}}{x_{0}}\right)[L(1),Y(a,x_{1})Y(b,x_{2})]\\
&&-\Res_{x_{1}}x_{0}^{-1}\delta\left(\frac{x_{2}-x_{1}}{-x_{0}}\right)[L(1),Y(b,x_{2})Y(a,x_{1})]\\
&=&\Res_{x_{1}}x_{0}^{-1}\delta\left(\frac{x_{1}-x_{2}}{x_{0}}\right)
\left(Y(A(x_{1})a,x_{1})Y(b,x_{2})+Y(a,x_{1})Y(A(x_{2})b,x_{2})\right)\\
&&-\Res_{x_{1}}x_{0}^{-1}\delta\left(\frac{x_{2}-x_{1}}{-x_{0}}\right)
\left(Y(A(x_{2})b,x_{2})Y(a,x_{1})+Y(b,x_{2})Y(A(x_{1})a,x_{1})\right)\\
&=&\Res_{x_{1}}x_{2}^{-1}\delta\left(\frac{x_{1}-x_{0}}{x_{2}}\right)
\left[Y\left(Y(A(x_{1})a,x_{0})b,x_{2}\right)+Y\left(Y(a,x_{0})A(x_{2})b,x_{2}\right)\right]\\
&=&Y\left(Y(A(x_{2}+x_{0})a,x_{0})b,x_{2}\right)+Y\left(Y(a,x_{0})A(x_{2})b,x_{2}\right).
\end{eqnarray*}
On the other hand, we have
\begin{eqnarray*}
&&Y\left((L(1)+2x_{2}L(0)+x_{2}^{2}L(-1))Y(a,x_{0})b,x_{2}\right)\\
&=&Y(Y(a,x_{0})L(1)b,x_{2})+Y(Y((L(1)+2x_{0}L(0)+x_{0}^{2}L(-1))a,x_{0})b,x_{2})\\
&&+2x_{2}Y(Y(a,x_{0})L(0)b,x_{2})+2x_{2}Y(Y((L(0)+x_{0}L(-1))a,x_{0})b,x_{2})\\
&&+x_{2}^{2}Y(Y(a,x_{0})L(-1)b,x_{2})+x_{2}^{2}Y(Y(L(-1)a,x_{0})b,x_{2})\\
&=&Y(Y((L(1)+2(x_{2}+x_{0})L(0)+(x_{2}+x_{0})^{2}L(-1))a,x_{0})b,x_{2})\\
&&+Y\left(Y(a,x_{0})(L(1)+2x_{2}L(0)+x_{2}^{2}L(-1))b,x_{2}\right)\\
&=&Y\left(Y(A(x_{2}+x_{0})a,x_{0})b,x_{2}\right)+Y\left(Y(a,x_{0})A(x_{2})b,x_{2}\right).
\end{eqnarray*}
Thus
\begin{eqnarray*}
[L(1),Y(Y(a,x_{0})b,x_{2})]=Y\left(A(x_{2})Y(a,x_{0})b,x_{2}\right).
\end{eqnarray*}
Then it follows that $[L(1),Y(v,x)]=Y(A(x)v,x)$ for all $v\in V$.
Therefore, $V$ is a M\"obius vertex algebra.
\end{proof}

\begin{lemma}[C]
Let $\A$ be a commutative and associative algebra equipped with a symmetric
associative bilinear form $\lan\cdot,\cdot\ran$ and let $\widetilde{L}(\A)=\A\otimes \bC[t,t^{-1}]\oplus \bC \c$ be the corresponding Lie algebra with $\c$ central and with
$$[a\otimes t^{m},b\otimes t^{n}]=(m-n)(ab\otimes t^{m+n-1})+\frac{1}{12}m(m-1)(m-2)\lan a,b\ran \delta_{m+n-2,0}\c$$
for $a,b\in \A,\ m,n\in \bZ$.
Then $sl_{2}$ acts on $\widetilde{L}(\A)$ as a Lie algebra of derivations with
\begin{eqnarray*}
&&L(-1)\cdot (a\otimes t^{n})=-\frac{d}{dt}(a\otimes t^{n})=-n(a\otimes t^{n-1}),\\
&&L(0)\cdot (a\otimes t^{n})=\left(1-t\frac{d}{dt}\right)(a\otimes t^{n})=(1-n)(a\otimes t^{n}),\\
&&L(1)\cdot (a\otimes t^{n})=\left(2t-t^{2}\frac{d}{dt}\right)(a\otimes t^{n})=(2-n)(a\otimes t^{n+1})
\end{eqnarray*}
for $a\in \A,\ n\in \bZ$ and with $sl_{2}\cdot \c=0$.
\end{lemma}

\begin{proof} First, we have the required commutation relations
\begin{eqnarray*}
&&\left[2t-t^{2}\frac{d}{dt},-\frac{d}{dt}\right]=\left[\frac{d}{dt},2t-t^{2}\frac{d}{dt}\right]
=2-2t\frac{d}{dt},\\
&&\left[1-t\frac{d}{dt},-\frac{d}{dt}\right]=\left[\frac{d}{dt},1-t\frac{d}{dt}\right]
=-\frac{d}{dt},\\
&&\left[1-t\frac{d}{dt},2t-t^{2}\frac{d}{dt}\right]
=-\left[t\frac{d}{dt},t\left(2-t\frac{d}{dt}\right)\right]
=-t\left(2-t\frac{d}{dt}\right)=-\left(2t-t^{2}\frac{d}{dt}\right).
\end{eqnarray*}
For $a,b\in \A,\ m,n\in \bZ$, we have
\begin{eqnarray*}
&&[L(-1)(a\otimes t^{m}),b\otimes t^{n}]+[a\otimes t^{m},L(-1)(b\otimes t^{n})]\\
&=&-m[a\otimes t^{m-1},b\otimes t^{n}]-n[a\otimes t^{m},b\otimes t^{n-1}]\\
&=&-m(m-1-n)(ab\otimes t^{m+n-2})-\frac{1}{12}m(m-1)(m-2)(m-3)\delta_{m+n-3,0}\lan a,b\ran\c\\
&&-n(m-n+1)(ab\otimes t^{m+n-2})-\frac{1}{12}nm(m-1)(m-2)\delta_{m+n-3,0}\lan a,b\ran\c\\
&=&(m-n)(1-m-n)(ab\otimes t^{m+n-2})\\
&=&L(-1)[a\otimes t^{m},b\otimes t^{n}],
\end{eqnarray*}
\begin{eqnarray*}
&&[L(1)(a\otimes t^{m}),b\otimes t^{n}]+[a\otimes t^{m},L(1)(b\otimes t^{n})]\\
&=&(2-m)[a\otimes t^{m+1},b\otimes t^{n}]+(2-n)[a\otimes t^{m},b\otimes t^{n+1}]\\
&=&(2-m)(m+1-n)(ab\otimes t^{m+n})+\frac{1}{12}(2-m)(m^{3}-m)\delta_{m+n-1,0}\lan a,b\ran\c\\
&&+(2-n)(m-n-1)(ab\otimes t^{m+n})+\frac{1}{12}(2-n)m(m-1)(m-2)\delta_{m+n-1,0}\lan a,b\ran\c\\
&=&(m-n)(3-m-n)(ab\otimes t^{m+n})\\
&=&L(1)[a\otimes t^{m},b\otimes t^{n}].
\end{eqnarray*}
This proves that $L(-1)$ and $L(1)$ act as derivations. As $[L(1),L(-1)]=2L(0)$, $L(0)$ also acts as a derivation.
\end{proof}

\begin{center}
{\bf Acknowledgements}
\end{center}

We gratefully acknowledge the partial financial support from NSFC (11271202, 11221091) and SRFDP (20120031110022) for C. Bai; NSA
(H98230-11-1-0161) and NSFC (11128103) for H. Li; NSFC (11101285,11371134), the
Shanghai Natural Science Foundation (11ZR1425900), and ZJNSF (LQ12A01005) for Y. Pei. Part of this work was done while the third author was visiting Rutgers University-Camden and he wishes to thank the Mathematics Department for its support and hospitality.

\end{document}